%% file: arxiv.tex
\documentclass{ceurart}

\usepackage{listings}
\usepackage{amsthm}
\usepackage{enumitem}
\usepackage{upquote}

\newcommand{\set}[1]{\left\{#1\right\}}

\newcommand{\ub}[1]{\underline{\textbf{#1}}}
\newcommand{\ur}[1]{\underline{\text{#1}}}

\newcommand{\rb}[1]{{\textbf{#1}}}
\newcommand{\infinity}{{${\infty}$}}

\newcommand{\myitemize}[2]{\begin{itemize}[leftmargin=#1,nosep]#2\end{itemize}
\vspace{1mm}}

\newtheorem{theorem}{Theorem}[section]
\newtheorem{lemma}{Lemma}[section]

\newtheorem{proposition}{Proposition}[section]

\newtheorem{observation}{Observation}[section]
\newtheorem{conjecture}{Conjecture}[section]
\newtheorem{example}{Example}[section]

\newcommand{\AC}{Auto\-Case}
\newcommand{\SymPy}{\texttt{SymPy}}
\newcommand{\ZT}{\texttt{Z3}}

\begin{document}
\copyrightyear{2025}
\copyrightclause{Copyright for this paper by its authors.
  Use permitted under Creative Commons License Attribution 4.0
  International (CC BY 4.0).}

\conference{10th International Workshop on Satisfiability Checking and Symbolic Computation, August 2, 2025, Stuttgart, Germany}

\title{Symbolic Sets for Proving Bounds on Rado Numbers}

\author[1]{Tanbir Ahmed}[
  email=tanbir@gmail.com,
  url=https://tanbir.github.io/
]
\author[1]{Lamina Zaman}[
  email=zamanl@uwindsor.ca
]
\author[1]{Curtis Bright}[
  orcid=0000-0002-0462-625X,
  email=cbright@uwindsor.ca,
  url=https://cs.curtisbright.com/
]

\address[1]{School of Computer Science, University of Windsor, Canada}

\begin{abstract}
  Given a linear equation $\cal E$ of the form $ax+by=cz$ where $a$, $b$, $c$ are positive integers, the $k$-colour \emph{Rado number} $R_k({\cal E})$ is the smallest positive integer $n$, if it exists, such that every $k$-colouring of the
  positive integers $\{1,2,\dotsc,n\}$ contains a monochromatic solution to ${\cal E}$.
  In this paper, we consider $k=3$ and the linear equations $ax+by=bz$ and $ax+ay=bz$.
  Using SAT solvers, we compute a number of previously unknown Rado numbers corresponding to these equations. 
  We prove new general bounds on Rado numbers inspired by the satisfying assignments discovered by the SAT solver.
  Our proofs require extensive case-based analyses that are difficult to check
  for correctness by hand, so we automate checking the correctness of our proofs
  via an approach which makes use of a new tool we developed
  with support for operations on symbolically-defined sets---%
  e.g., unions or
  intersections of sets of the form $\{f(1),f(2),\dotsc,f(a)\}$
  where $a$ is a \emph{symbolic} variable and $f$ is a function
  possibly dependent on $a$.
  No computer algebra system that we are aware of currently has
  sufficiently capable support for symbolic sets,
  leading us to develop
  a tool supporting symbolic sets using the Python symbolic computation
  library \SymPy\ coupled with the Satisfiability Modulo Theories solver \ZT.
  \end{abstract}


\maketitle

\section{Introduction}
The study of Rado numbers \cite{rado1933} and colouring problems lies at the heart of Ramsey theory and additive combinatorics. 
In this paper, we present a number of new values and bounds on Rado numbers,
as well as an automated framework for proving lower bounds on Rado numbers using symbolic partitioning and satisfiability checking.

\subsection{Combinatorial context}
For integers $a$, $b$, let $[a,b]$ denote the set of integers $\set{x: a\leq x\leq b}$. 
For a positive integer $k\geq 2$, a system of linear equations is \emph{$k$-regular}
if there exists a positive integer $n$ such that for every $k$-colouring of 
$[1,n]$, there exists a monochromatic solution to the system.
A system is \emph{regular} if it is $k$-regular for every integer $k\geq 2$.
The following theorem by Rado \cite{rado1933} gives necessary and sufficient conditions for when a linear equation is regular.

\begin{theorem}[Rado \cite{rado1933}]\label{th-rado}
For nonzero integers $a_1$, $a_2$, $\dotsc$, $a_m$ and integer $c$, 
let $s=\sum_{i=1}^ma_i$. Then:
\myitemize{11pt}{
\item The equation $\sum_{i=1}^ma_ix_i=0$ is regular if and only if there 
exists a nonempty set $D\subseteq [1,m]$ such that $\sum_{d\in D}a_d=0$.
\item The equation $\sum_{i=1}^ma_ix_i=c$ is regular if and only if one 
of the following two conditions holds:
$(a)$ $c/s\in {\mathbb Z}^+$;
$(b)$ $c/s\in {\mathbb Z}^-$ and $\sum_{i=1}^ma_ix_i=0$ is regular.
}
\end{theorem}

For nonzero integers $a_1$, $a_2$, $\dotsc$, $a_m$ and integer~$c$, let the linear equation $\sum_{i=1}^{m-1}a_ix_i+c=a_mx_m$ be represented by
${\cal E}(m, c; a_1,a_2,\ldots,a_m)$.
For a linear equation ${\cal E}(m,c; a_1,a_2,\ldots,a_m)$ and a positive integer $k$, the $k$-colour
\emph{Rado number} $R_k({\cal E}(m,c; a_1,a_2,\ldots,a_m))$ is defined as the smallest positive integer
$n$, if it exists, from the definition of $k$-regularity demonstrating that the equation ${\cal E}(m,c; a_1,a_2,\ldots,a_m)$ is $k$-regular.  Otherwise, we say that
$R_k({\cal E}(m,c; a_1,a_2,\ldots,a_m))$ is infinite if there is a $k$-colouring of the positive integers with no monochromatic solution to ${\cal E}(m,c; a_1,a_2,\ldots,a_m)$.  
Surprisingly, not much is known about the properties of Rado numbers for $k=3$.
In 1995, Schaal~\cite{schaal1995} proved that the 3-colour Rado numbers $R_3({\cal E}(3,c;1,1,1))$ are always finite and $R_3({\cal E}(3,c;1,1,1)) = 13c+14$ for $c\geq 0$.
In 2015, Adhikari et al.~\cite{adhikari2015} proved exact formulas for \(R_3({\cal E}(4,c;1,1,1,1))\) and \(R_3({\cal E}(5,c;1,1,1,1,1))\) with \(c\geq 0\).
In 2022, Chang, De Loera, and Wesley~\cite{chang2022} proved
\myitemize{11pt}{
\item $R_3({\cal E}(3, 0; 1, -1, a-2)) = a^3-a^2-a-1$ for $a\geq 3$ 
(previously conjectured by Myers \cite{myers2015}),
\item $R_3({\cal E}(3, 0, a, -a, (a-1))) = a^3+(a-1)^2$ for $a\geq 3$, and
\item $R_3({\cal E}(3, 0;a, -a, b)) = a^3$ for $b\geq 1$, $a\geq b+2$, and $\gcd(a,b)=1$.
}
They also determined exact values of the Rado numbers
\myitemize{11pt}{
    \item \( R_3({\cal E}(3, 0;a, -a, b)) \) for \( 1 \leq a, b \leq 15 \),
    \item \( R_3({\cal E}(3, 0;a, a, b)) \) for \( 1 \leq a, b \leq 10 \),
    \item \( R_3({\cal E}(3, 0;a, b, c)) \) for \( 1 \leq a, b, c \leq 6 \),
}
as well as the following theorem.
\begin{theorem}\label{th:chang2022}
$R_3({\cal E}(3, 0; 1, 1, a)) = \infty$ for $a\geq 4$ and $R_3({\cal E}(3, 0; a, a, 1)) = \infty$ for $a\geq 2$.
\end{theorem}

\subsection{Our contributions}
In this paper, we have investigated values and bounds of Rado numbers for the equations ${\cal E}(3,0; a,a,b)$ and ${\cal E}(3,0; a,b,b)$.
We have computed a number of previously unknown exact values using SAT solvers, extending the results presented by Chang, De Loera, and Wesley~\cite{chang2022}
to the values of
\myitemize{11pt}{
\item $R_3({\cal E}(3,0;a,b,b))$ for $1\leq b\leq 15$ and $1\leq a\leq 30$, and
\item $R_3({\cal E}(3,0;a,a,b))$ for $1\leq a\leq 15$ and $1\leq b\leq 25$.
}  
Note that we may assume that $a$ and $b$ are coprime, since if $a$ and $b$ share a common factor $g$
we may divide the coefficients of the equation ${\cal E}$ by $g$ without changing its solutions, e.g.,
$R_3({\cal E}(3,0;a,b,b))=R_3({\cal E}(3,0;a/g,b/g,b/g))$.
Guided by the satisfying assignments generated by the SAT solver in the process of computing these values, we proved the following
two theorems.
\begin{theorem}\label{th:lb-abb}
  For coprime positive integers $a$, $b$ with $a > b\geq 3$ and $a^2+a+b > b^2+ba$,
    \[
    {R_3}({\cal E}(3, 0; a, b, b)) \geq a^3 + a^2 + (2b+1)a + 1.
    \]
\end{theorem}
\begin{theorem}\label{th:lb-aaap1}
    For odd integers $a \geq 7$,
    \[
    {R_3}({\cal E}(3, 0; a, a, a+1)) \geq a^3(a+1).
    \]
\end{theorem}
Based on the computational data we generated and the above lower bounds, we propose the following conjectures.
  \begin{conjecture}\label{conj-abb}
    For coprime positive integers $a$, $b$ with $a > b\geq 3$ and $a^2+a+b > b^2+ba$,
    \[
    {R_3}({\cal E}(3, 0; a, b, b)) = a^3 + a^2 + (2b+1)a + 1.
    \]
    \end{conjecture}
  \begin{conjecture}\label{conj-aab-odd}
      For coprime positive integers $a$, $b$, with $a$ odd, if\/ $3\leq a<b\leq 2a-1$ then
      \[
      R_3({\cal E}(3, 0; a, a, b)) = a^3b.
      \]
  \end{conjecture}

The proofs of Theorems \ref{th:lb-abb} and \ref{th:lb-aaap1} are automatically verified using a symbolic analyzer we developed in Python called 
\AC.\footnote{\href{https://github.com/laminazaman/RadoNumbers/tree/main/AutoCase}{github.com/laminazaman/RadoNumbers/tree/main/AutoCase}}
The motivation has been to avoid errors in the intricate case-based analysis which is often routine and repetitive.
For example, in order to verify Theorem~\ref{th:lb-aaap1}, AutoCase takes as input
three symbolically-defined subsets $R$, $G$, $B$ of the integers $[1,a^3(a+1)-1]$ (where $a\geq7$ is a symbolic odd integer).
\AC\ verifies that the union of $R$, $G$, and $B$ is the set $[1,a^3(a+1)-1]$ and that $R$, $G$, $B$ are mutually disjoint---thereby
confirming that $R$, $G$, $B$ form a 3-colouring of the positive integers less than $a^3(a+1)$.  Finally, \AC\ confirms
that there are no monochromatic solutions of the equation $ax+ay=(a+1)z$, i.e., the set $\bigcup_{ax+ay=(a+1)z}\{(x,y,z)\}$
is disjoint with $R^3$, $G^3$, and $B^3$.
We stress that the sets defining the colouring are defined \emph{symbolically}, e.g., $\bigcup_{i,j,k}\{f(i,j,k)\}$
where $f$ is a function (possibly depending on~$a$), the upper and lower bounds on the indices $i$, $j$, $k$
may depend on $a$, and the elements of the set may be filtered by divisibility predicates such as specifying that
all elements are divisible by $a$.

Although it would be easy to perform the necessary
set operations in a typical computer algebra system like Maple, Mathematica, and SageMath
if $a$ was a known fixed integer, in our application $a$ is a symbolic variable,
and we are aware of no computer algebra system supporting
the operations on symbolic sets we require in this paper.  
Maple and Mathematica seem to lack the ability to even express the set $[1,a]$ when $a$ is symbolic.
SageMath does have support for a ``ConditionSet'' allowing sets like $[1,a]$ to be formulated, but
the operations supported by ConditionSets was too limited for our purposes.
In \AC, operations on symbolic sets are performed by employing a combination of the Python symbolic computation
library \SymPy~\cite{sympy} and the SMT (SAT modulo theories) constraint solver \ZT~\cite{deMoura2008}.

Our work therefore fits into the ``SC-Square'' paradigm of combining satisfiability checking
with symbolic computation~\cite{abraham2016,Bright2022}.  Over the past decade, the SC-Square
community has combined the tools of satisfiability checking (e.g., SAT and SMT solvers)
with the tools of symbolic computation (e.g., computer algebra systems and libraries) in order
to make progress on problems benefiting from both the search and learning of SAT/SMT solvers
and the mathematical sophistication of computer algebra.  As just a single example,
computer algebra libraries are able to detect if two mathematical objects are isomorphic.
Augmenting a SAT solver with isomorphism detection can dramatically improve
its efficiency on mathematical problems like proving the nonexistence of an order ten projective
plane~\cite{Bright2021} or enumerating various kinds of graphs up to isomorphism,
something that can be done by SAT modulo symmetries~\cite{Kirchweger2023} or SAT+CAS~\cite{Li2024} solvers.

\section{Computational Results}\label{sec:computational-results}

One of the key challenges in Ramsey theory on the integers is the scarcity of data, as generating individual data points is an exceptionally difficult task. However, modern computational tools, including SAT solvers, have alleviated some of these computational challenges. During the last two decades, SAT solvers have been employed to compute Ramsey-type numbers at different scales and capacities. Some examples are the computation of the values and bounds of van der Waerden numbers (Kouril and Paul~\cite{kouril2006}, Herwig et al.~\cite{herwig2007}, Ahmed~\cite{ahmed2009, ahmed2010, ahmed2011, ahmed2013}, Ahmed et al.~\cite{aks2014}) and Schur numbers (Ahmed and Schaal~\cite{as2015}, Ahmed et al.~\cite{abrs2023a}).  A notable success of SAT solvers in this area of mathematics is the computation of the fifth Schur number by Heule~\cite{heule2018}. 

Recently, Chang, De Loera, and Wesley~\cite{chang2022} applied SAT solvers to compute Rado numbers. Building on this progress, we used SAT solvers to determine new exact values of Rado numbers and to uncover patterns in the colourings that avoid monochromatic solutions to the linear equation under investigation. These patterns
provide general lower bounds on Rado numbers and we employed methods from symbolic computation and satisfiability checking to automate the verification of these lower bounds.

\subsection{The satisfiability (SAT) problem}

A {\em literal} is a Boolean variable (say $x$) or its negation (denoted $\bar x$). A {\em clause} is a logical disjunction of literals. A {\em formula} is in {\em Conjunctive Normal Form (CNF)} if it is a logical conjunction of clauses.

A {\em truth assignment} is a mapping of each variable in its domain to true or false.
A truth assignment {\em satisfies a clause} if it maps at least one of its literals to true and the assignment {\em satisfies a formula} if it satisfies each of its clauses.
A formula is called {\em satisfiable} if it is satisfied by at least one truth assignment and otherwise it is called {\em unsatisfiable}.
The problem of recognizing satisfiable formulas is known as {\em the satisfiability problem}, or SAT for short.

\subsection{Encoding Rado numbers as SAT problems}

Given positive integers $n$ and $k$ and a linear equation $\cal E$, we now construct a formula in conjunctive normal form that is 
satisfiable if and only if there exists a $k$-colouring of $[1,n]$ avoiding monochromatic solutions of ${\cal E}$.
Therefore, if our formula is satisfiable then  
$R_k({\cal E})>n$ and if our formula is unsatisfiable then $R_k({\cal E})\leq n$.
\paragraph{Variables:} Variables are denoted by $v_{i,j}$ for $0 \leq i \leq k-1$ and $1\leq j\leq n$, such that $v_{i,j}$ is true if and only if colour $i$ is assigned to integer $j$. There are $nk$ variables in the formula. 
\paragraph{At least one colour is assigned to every integer:} For each integer $j\in [1,n]$, the clause \[(v_{0,j}\vee v_{1,j}\vee\dotsb\vee v_{k-1,j})\] ensures at least one colour $i\in [0,k-1]$ is assigned to $j$. 
\paragraph{At most one colour is assigned to every integer:} For each integer $j\in [1,n]$, the clauses \[\bigwedge_{0\leq i_1 < i_2\leq k-1}(\bar{v}_{i_1,j} \vee \bar{v}_{i_2,j})\] ensure at most one colour $i\in [0,k-1]$ is assigned to $j$.
\paragraph{There is no monochromatic solution to ${\cal E}$:} For each colour $i\in [0,k-1]$ and for each solution $(x_1,x_2,\ldots,x_m)$ to ${\cal E}$, the clause
\[(\bar{v}_{i,x_1} \vee \bar{v}_{i,x_2} \vee \cdots \vee \bar{v}_{i,x_m})\] ensures
the solution $(x_1,x_2,\dotsc,x_m)$ is not monochromatic in colour~$i$. Let ${\cal S}_{{\cal E}, n}$ denote the set of all solutions $(x_1,x_2,\ldots,x_m)$ in $[1,n]$ that satisfy the equation ${\cal E}$. There will be $k\cdot |{\cal S}_{{\cal E}, n}|$ such clauses. 
We may compute $|\mathcal{S}_{\mathcal{E}, n}|$ by summing over all possible values of \( (x_1, x_2, \dots, x_{m-1}) \) in the range \( [1, n] \) via
\[
|\mathcal{S}_{\mathcal{E}, n}| = \sum_{x_1=1}^{n} \sum_{x_2=1}^{n} \dots \sum_{x_{m-1}=1}^{n} \left[\frac{\sum_{i=1}^{m-1} a_i x_i}{a_m} \in [1, n]\right] ,
\]
where summand uses Iverson bracket notation to ensure that the remaining variable \( x_m \) (determined by the equation) falls within \( [1, n] \).
This gives $|\mathcal{S}_{\mathcal{E}, n}|\leq n^{m-1}$, since \( [\cdot] \leq 1 \).
\paragraph{Breaking permutation symmetry of colours:}
So far, no distinction has been made amongst the~$k$ colours.  This artificially increases the size
of the search space by a factor of $k!$, the number of permutations on $k$ colours.
To avoid having the SAT solver explore all $3!$ colour permutations on $\{0,1,2\}$ during solving,
we ensure that colours are introduced in ascending order, i.e., colour~$0$ appears before colour~$1$
and colour~$1$ appears before colour~$2$.
To ensure that colour~$0$ appears in the first position, we add the unit clause $v_{0,1}$.
To ensure that colour~$1$ appears before colour~$2$, we
add the clause $(\bigvee_{i=1}^{j-1}\bar v_{0,i})\lor\bar v_{2,j}$
for $j=2$, $\dotsc$, $R_1(\cal E)$.  These say that
colour 2 cannot immediately follow an initial colouring
consisting only of colour 0s, and thus the first colour other than 0
must be 1.  One can work out the exact value of $R_1(\cal E)$ by hand for the cases we consider in this paper,
as in the following lemma whose proof we provide in appendix~\ref{sec:r1proof}.
\begin{lemma}\label{lem:r1value}
If\/ $a$ and\/ $b$ are positive coprime integers then\/ $R_1({\cal E}(3,0;a,b,b))=\max(a+1,b)$,
$R_1({\cal E}(3,0;a,a,1))=2a$, and\/ $R_1({\cal E}(3,0;a,a,b))=\max(a,\lceil b/2\rceil)$ when\/ $b>1$.
\end{lemma}

\subsection{Efficient generation of integer solutions}
To efficiently generate all integer solutions $(x, y, z)$ within the interval \([1, n]\) for
$ax + by = cz$,
we iterate \( i \) from \( 1 \) to \( n \) and construct three linear Diophantine equations
\begin{align*}
a x &= cz - by \quad \text{where } x = i, \\
b y &= cz - ax \quad \text{where } y = i, \\
c z &= ax + by \quad \text{where } z = i.
\end{align*}
The goal is to independently find solutions for each equation.  

\myitemize{11pt}{
\item Step 1 (Solving a linear Diophantine equation).
To begin, we solve one of the above equations (say $ax+by=ci$) for unknowns $(x,y)$ using the Extended Euclidean Algorithm, providing a particular integer solution \((x_0, y_0)\).
The general solution is
\[
x = x_0 + \frac{b k}{\gcd(a, b)} \quad\text{and}\quad y = y_0 - \frac{a k}{\gcd(a, b)},
\]
where \( k \) is a parameter that runs over all integers to generate all solutions.

\item Step 2 (Restricting to the interval $[1,n]$).
Since we are only interested in solutions within \([1, n]\), we impose the constraints
\[
1 \leq x_0 + \frac{b k}{\gcd(a, b)} \leq n \;\text{ and }\; 1 \leq y_0 - \frac{a k}{\gcd(a, b)} \leq n .
\]
These inequalities provide upper and lower bounds on \( k \), ensuring that both \( x \) and \( y \) remain in the valid range. 

\item Step 3 (Iterating over valid values of \( k \)).
Once the feasible range for \( k \) is determined, we iterate over all valid values of \( k \) and 
solve for \( x \) and \( y \) at each step, adding (\( x \), \( y \), \( i \)) to our list of solutions.
\item Finally, we repeat these steps for the remaining two equations $a i = cz - by$ and $b i = cz - ax$.
}

\subsection{SAT solvers and computation resources}
For this work, we used the SAT solvers CaDiCaL~\cite{biere2024cadical} and Kissat~\cite{biere2024kissat}. Initially, we used CaDiCaL integrated with PySAT~\cite{imms-sat18} due to its incremental solving capabilities which enable the solver to reuse learned information across instances. Our approach starts with a counter \(n = 1\). For each \(n\), we add the corresponding clauses to the solver. If the instance is satisfiable, we increment \(n\) and continue incrementally. This process repeats until we encounter the first unsatisfiable instance, at which point the corresponding \(n\) is identified as the Rado number.

Once a Rado number was determined (or suspected to be known based on
the patterns we observed), we transitioned to using a non-incremental solver and generated two instances:
one for \(n - 1\), the last satisfiable instance, and one for \(n\), the first unsatisfiable instance.
We generated clauses in DIMACS format, wrote them to two CNF files, and used Kissat to solve both instances.
To be consistent in our timings, all instances were ultimately solved non-incrementally using Kissat.
Our largest instance, the unsatisfiable instance showing $R_3({\cal E}(3,0;15,15,8))=97875$,
contained 293,625 variables and 255,884,401 clauses and took 58.8 hours to solve.

Altogether, generating the SAT instances took around 11 hours,
the total CPU time used to solve the unsatisfiable instances (i.e., the instances that prove an upper bound) was around 251 hours,
and the total CPU time used on satisfiable instances (i.e., the instances that prove a lower bound) was around 53 hours.
Our computations were performed on the Compute Canada cluster \href{https://docs.alliancecan.ca/wiki/Cedar}{\texttt{Cedar}}, utilizing Intel E5-2683 v4 Broadwell processors running at 2.1 GHz.

\setlength{\extrarowheight}{1pt}

\input{table1.tex}

\input{table2.tex}

\subsection{Some new exact values of $R_3({\cal E}(3,0;a,b,b))$}
We have computed $R_3({\cal E}(3,0;a,b,b))$ for $1\leq b\leq 15$ and $1\leq a\leq 30$
and given the results in Table~\ref{tab:abb}.
By Rado's theorem \ref{th-rado}, all elements in Table~\ref{tab:abb} are finite.
Altogether, the total CPU time on unsatisfiable instances was 21,790 seconds,
and the total CPU time on satisfiable instances was 13,516 seconds, 
amounting to 9.81 total hours. A maximum of 4.3 GiB of memory was used across all instances.

\subsection{Some new exact values of $R_3({\cal E}(3,0;a,a,b))$}
By Theorem \ref{th:chang2022}, $R_3({\cal E}(3, 0;a, a, 1))$ is infinite for $a\geq 2$ and $R_3({\cal E}(3, 0;1, 1, b))$ is infinite for $b\geq 4$.
Also, by Chang et al.~\cite[Lemma 3.1]{chang2022}, $R_3({\cal E}(3, 0;a, a, b))$ is infinite
if $2a \leq a^2/b$ or $2a \leq \sqrt{ab}$.
They provided the values of $R_3({\cal E}(3,0;a,a,b))$ for $1\leq a, b\leq 10$ and also for $3\leq a\leq 6$ and $11\leq b\leq 20$. 
We have extended them for $1\leq a\leq 15$ and $1\leq b\leq 25$ and
reported these numbers in Table~\ref{tab:aab}.
The instances in this table proving upper bounds required a CPU time of 882,211 seconds to solve, while the instances proving lower bounds required 176,375 seconds to solve, 
amounting to 294.05 total hours.  A maximum of 26.3 GiB of memory was used across all instances.

\subsection{Some new patterns for $R_3({\cal E}(3,0;a,b,b))$}
The data on $R_3({\cal E}(3, 0; a, b, b))$ presented in Table~\ref{tab:abb} inspired us to make some general observations and a conjecture described below.

\begin{observation}\label{obs:abb-lb}
  For coprime integers\/ $a$ and\/ $b$, the data in Table~\ref{tab:abb} reveals the following.
  \begin{center}
  \begin{tabular}{cc}
  \begin{tabular}{|c|c|c|}
  \hline
  $b$ & Range of\/ $a$ & $R_3({\cal E}(3, 0; a, b, b))$ \\
  \hline
  $2$ & $[7,30]$ & $a^3+a^2+5a+1$ \\
  \hline
  $3$ & $[4,30]$ & $a^3+a^2+7a+1$ \\
  \hline
  $4$ & $[7,30]$ & $a^3+a^2+9a+1$ \\
  \hline
  $5$ & $[7,30]$ & $a^3+a^2+11a+1$ \\
  \hline
  $6$ & $[11,30]$ & $a^3+a^2+13a+1$ \\
  \hline
  $7$ & $[11,30]$ & $a^3+a^2+15a+1$ \\
  \hline
  $8$ & $[13,30]$ & $a^3+a^2+17a+1$ \\
  \hline
  \end{tabular} &
  \begin{tabular}{|c|c|c|}
  \hline
  $b$ & Range of\/ $a$ & $R_3({\cal E}(3, 0; a, b, b))$ \\
  \hline
  $9$ & $[14,30]$ & $a^3+a^2+19a+1$ \\
  \hline
  $10$ & $[17,30]$ & $a^3+a^2+21a+1$ \\
  \hline
  $11$ & $[17,30]$ & $a^3+a^2+23a+1$ \\
  \hline
  $12$ & $[19,30]$ & $a^3+a^2+25a+1$ \\
  \hline
  $13$ & $[20,30]$ & $a^3+a^2+27a+1$ \\
  \hline
  $14$ & $[23,30]$ & $a^3+a^2+29a+1$ \\
  \hline
  $15$ & $[26,30]$ & $a^3+a^2+31a+1$ \\
  \hline
  \end{tabular}
  \end{tabular}
  \end{center}
\end{observation}

The formulas presented above were derived through curve fitting. Our analysis began by identifying patterns within each column of data. Some numbers exhibited a constant difference, suggesting a linear relationship, while others showed a constant third difference, indicating a cubic relationship.
We constructed systems of linear and cubic equations based on these observations.
We then applied Gaussian--Jordan elimination to solve these systems, obtaining the best-fitting formulas describing the sequence and
ultimately arriving at the following conjecture.

\begin{conjecture}%
For coprime positive integers $a$ and $b$ with $a^2+a+b > b^2+ba$ and $a > b\geq 3$,
${R_3}({\cal E}(3, 0; a, b, b)) = a^3 + a^2 + (2b+1)a + 1$.
\end{conjecture}

Before proceeding, we note
the equation ${\cal E}(3, 0; a, -a, b)$ can be transformed into ${\cal E}(3, 0; b, a, a)$ by permuting variables. This equivalence
along with a theorem of Chang, De Loera, and Wesley~\cite{chang2022} provides a closed-form
expression for $R_3({\cal E}(3, 0; a, 1, 1))$.

\begin{proposition}\label{prop:r3a11}
$R_3({\cal E}(3, 0; a, 1, 1)) = a^3+5a^2+7a+1$
for $a\geq 1$.
\begin{proof}
By permuting variables we have $R_3({\cal E}(3, 0; a, 1, 1)) = R_3({\cal E}(3, 0; 1, -1, a))$. By~\cite[Thm~1.2]{chang2022}, for $m\geq 3$, 
$R_3({\cal E}(3, 0; 1, -1, m-2)) = m^3-m^2-m-1$. Replacing $a$ with $m-2$, we get
for $a\geq 1$, $R_3({\cal E}(3, 0; 1, -1, a)) = a^3+5a^2+7a+1$.
Hence, $R_3({\cal E}(3, 0; a, 1, 1)) = a^3+5a^2+7a+1.$
\end{proof}
\end{proposition}

\subsection{Some new patterns for $R_3({\cal E}(3,0;a,a,b))$}
The data on $R_3({\cal E}(3, 0; a, a, b))$ presented in Table~\ref{tab:aab} has inspired the following conjecture and some general observations
described below.

\begin{conjecture}%
For coprime positive integers $a$ and $b$, if $a$ is odd, then
for $3\leq a<b\leq 2a-1$,
$R_3({\cal E}(3, 0; a, a, b)) = a^3b$.
\end{conjecture}

\begin{observation}
For odd positive integers \(7 \leq a \leq 15\), we have  
\myitemize{11pt}{
\item \(R_3({\cal E}(3,0; a, a, a-1)) = R_3({\cal E}(3,0; a, a, a+2)) = a^3(a+2)\);
\item if \(\gcd(a,a+3)=1\), then \(R_3({\cal E}(3,0; a, a, a-2)) = R_3({\cal E}(3,0; a, a, a+3)) = a^3(a+3)\);
\item \(R_3({\cal E}(3,0; a, a, \frac{a+1}{2})) = R_3({\cal E}(3,0; a, a, 2a-1)) = a^3(2a-1)\);
\item if \(\gcd(a,\frac{a+3}{2})=1\), then
\(R_3({\cal E}(3,0; a, a, \tfrac{a+3}{2})) = R_3({\cal E}(3,0; a, a, 2a-4))=a^3(2a-4)\).
}
\end{observation}

\begin{observation}
For coprime positive integers $a$ and $b$, 
\[
{R_3}({\cal E}(3, 0; a, a, b)) =
\begin{cases}
a^3b/2 & \text{if $a\in\set{6,10,14}$ and $a<b\leq 2a-1$,} \\
a^3b/4 & \text{if $a\in\set{12}$ and $a<b\leq 2a-1$.} \\
\end{cases}
\]
A general characterization of the cases with even $a$ is not possible at the moment due to lack of further data. 
\end{observation}

\section{General Results}\label{sec:general-results}
In Section~\ref{sec:computational-results}, we extended the known numerical values of Rado numbers corresponding to the equations ${\cal E}(3, 0; a, a, b)$ and ${\cal E}(3, 0; a, b, b)$ and observed several new bounds for these numbers.
In this section, we discuss and prove some of those bounds using \AC.
An overview of how \AC\ works is first provided in Section~\ref{sec:autocase},
and then as a demonstrative example we use \AC\ to prove
Proposition~\ref{prop:r3a11} in Section~\ref{sec:ac-example}.
Finally, we use \AC\ to prove Theorems~\ref{th:lb-abb} and~\ref{th:lb-aaap1}
in Sections~\ref{sec:thm1proof} and~\ref{sec:thm2proof}.

\subsection{Automating case-based analysis with \texttt{SymPy} and \texttt{Z3}}\label{sec:autocase}
Shallit \cite{shallit2021} proposed that traditional case-by-case analyses are better performed by automated search algorithms to enhance both efficiency and correctness. In that direction, we employ an automated verification approach (using our tool \AC) to verify the case-based proofs of Theorems~\ref{th:lb-abb} and~\ref{th:lb-aaap1}.
Our proof strategy relies on case-based analysis of solving linear equations whose variables lie in symbolic sets filtered through arithmetic and logical constraints.
This section describes how we automate this process by symbolically encoding the problem structure and constraint system using \texttt{SymPy} and performing satisfiability checking using the SMT solver \texttt{Z3}.

The input provided to \AC\ is the linear equation to be considered
along with a colouring of the integers $[1,N]$ specified by a collection
of symbolic sets as described in Section~\ref{sec:spec}.
\AC\ verifies that the colouring is valid by ensuring the size of the
symbolic sets sum to exactly $N$ (see Section~\ref{sec:symbolicsize}) and that
the sets are pairwise disjoint (see Section~\ref{sec:symbolicdisjoint}).
This verifies the provided symbolic sets provide a partition of $[1,N]$.
Finally, \AC\ verifies that there are no solutions
of the linear equation under consideration where all variables are
from symbolic sets of the same colour (see Section~\ref{sec:analyzer}).
Thus, the symbolic sets define a colouring of the integers $[1,N]$
that have no monochromatic solutions to the linear equation,
thereby proving a lower bound of $N+1$ on the Rado number for that linear equation.

\subsubsection{Input specification}\label{sec:spec}

The input to our automated system consists of:

\myitemize{11pt}{
    \item A symbolic linear equation $\cal E$, such as $ax + ay = (a+1)z$, where coefficients and assumptions on the coefficients (e.g., $a \geq 7$, $a$ odd) are symbolic.
    \item A symbolic $N$ corresponding to the lower bound on $R_k({\cal E})$, e.g., $N=a^3(a+1)-1$.
    \item A partition of the domain $[1, N]$ into named symbolic sets, each specified by
    \myitemize{11pt}{
        \item symbolic bounds (e.g., $[a, a^4 + a^3 - a]$),
        \item a format expression generating integers in the set (e.g., $a(a+1)i + aj + k$), and
        \item optional divisibility filters (e.g., integers in the set must be divisible by $a^2$ but not $a^3$).
    }
    \item A $k$-colouring of $[1,N]$ defined in terms of the above symbolic sets.
}
These inputs are defined using \texttt{SymPy} to facilitate symbolically: 
\myitemize{11pt}{
  \item Building and manipulating (e.g., simplifying, expanding, combining, substituting symbolically) interval bounds, format expressions, and divisibility properties.
  \item Calculating size, disjointness, and covering conditions.
  The union \(\bigcup_{i=1}^k S_i\) is considered a partition of $[1,N]$ if and only if 
  \(\bigcup_{i=1}^k S_i\) has size \(N\), the sets $S_i$ are pairwise disjoint, and \(\bigcup_{i=1}^k S_i \subseteq [1,N]\).
  \item Simplifying and structurally transforming constraints before they are fed to \ZT.
}

\subsubsection{Symbolic size computation using \texttt{SymPy}}\label{sec:symbolicsize}

\SymPy\ plays a central role in enabling symbolic reasoning, algebraic manipulation, and structured enumeration in the symbolic size analysis pipeline. The key uses of \SymPy\ can be itemized as follows:

\myitemize{11pt}{
  \item \textbf{Size with divisibility filters}:
    Let the underlying interval be defined as \([L, H]\).
    Let \(D\) and \(\mathit{ND}\) denote the sets of required and excluded divisors, respectively, and 
    consider \[S = \set{\,x \in [L, H]: \forall{d}\in D, d\mid x \text{ and } \forall{n}\in \mathit{ND}, n \nmid x\,}.\]
    For \(T\subseteq \mathit{ND}\), define
    \[S_T = \set{\,x \in [L, H]: \mathrm{lcm}(D\cup T)\mid x\,}.\]
    The number of elements \(x\in[L,H]\) divisible by all divisors in \(D\) and none in \(\mathit{ND}\) is
    \[|S| = \sum_{T\subseteq \mathit{ND}}(-1)^{|T|}|S_T| = 
      \sum_{T \subseteq \mathit{ND}} (-1)^{|T|}\left(
        \left\lfloor \frac{H}{\mathrm{lcm}(D \cup T)} \right\rfloor -
        \left\lfloor \frac{L - 1}{\mathrm{lcm}(D \cup T)} \right\rfloor\right) .
    \]
    This formula uses inclusion-exclusion to subtract overlapping exclusions in order to count elements satisfying the divisibility constraints.
    \texttt{SymPy} contributes with \texttt{floor}, \texttt{ceiling}, and \texttt{lcm} to enforce and combine arithmetic conditions.
    For example, consider the interval \([1, b^2a]\) with \(D=\set{b}\) and \(\mathit{ND}=\set{\mathstrut\smash{b^2}}\). Then
    the size of the resulting set is \[
    \left\lfloor \frac{b^2 a}{\mathrm{lcm}(b)} \right\rfloor -
    \left\lfloor \frac{b^2 a}{\mathrm{lcm}(b, b^2)} \right\rfloor =
    \left\lfloor \frac{b^2 a}{b} \right\rfloor -
    \left\lfloor \frac{b^2 a}{b^2} \right\rfloor = ba - a = a(b - 1) .
    \]    

    \item \textbf{Format-based symbolic summation}:
    Consider a symbolic set defined by a format expression \(
    x = f(i_1, i_2, \dotsc, i_k)\)
    where \(i_j \in [l_j, h_j]\). When $f$ is injective, the total size of the set is given by
    \[
    \left| \{ x = f(i_1, \dots, i_k) \mid l_j \leq i_j \leq h_j \} \right| =
    \sum_{i_1 = l_1}^{h_1} \sum_{i_2 = l_2}^{h_2} \cdots \sum_{i_k = l_k}^{h_k} 1.
    \]
    For example, the format expression \(x = a i + j\) where \(0 \leq i < m\) and \(1 \leq j \leq a\) defines a set 
    of size \(m a\). If some variables are symbolic (e.g., \( a \)), the size remains symbolic
    and can be simplified using algebraic rules. For example, \(\sum_{i=0}^{a} \sum_{j=1}^{i} 1 = a(a + 1)/2\).

    \item \textbf{Symbolic floor simplification}:
    Let \( a \) be a symbolic variable representing a positive integer, and consider the rational expression \( f(a) = {p(a)}/{q(a)} \), where \( p \) and~\( q \) are polynomials in \( a \).
    We aim to simplify \(\left\lfloor {p(a)}/{q(a)} \right\rfloor\)
    which can often be simplified using the degrees of the numerator and denominator polynomials.
    
    Consider \(\deg(p) < \deg(q)\).
    When the degree of the numerator \( p(a) \) is strictly less than that of the denominator \( q(a) \), the value of the rational expression
    \({p(a)}/{q(a)}\) has magnitude less than 1 for all sufficiently large \( a \), and is dominated by the leading terms of \( p \) and \( q \).    
    Its sign for large enough $a$ is therefore determined by the sign of
    \( \operatorname{LC}(p)/\operatorname{LC}(q)\)
    where \( \operatorname{LC}(\cdot) \) denotes the leading coefficient, so
    \[
    \left\lfloor \frac{p(a)}{q(a)} \right\rfloor =
    \begin{cases}
    0 & \text{if } {\mathrm{LC}(p)}/{\mathrm{LC}(q)} > 0, \\
    -1 & \text{if } {\mathrm{LC}(p)}/{\mathrm{LC}(q)} < 0,
    \end{cases}
    \]
    assuming that the bound on $a$ is large enough for this simplification to occur.
    For example, \(\left\lfloor -{1}/{a} \right\rfloor = -1\) if $a\geq1$, \(\left\lfloor {1}/{a} \right\rfloor = 0\),
    \(\left\lfloor (a+1)/a \right\rfloor = 1\), and \(\left\lfloor (a^2 + 1)/{a^3} \right\rfloor = 0\) if $a\geq2$, and
    \(\left\lfloor -(a + 5)/{a^2} \right\rfloor = -1\) if $a\geq3$. 
    
    \item \textbf{Unified symbolic simplification}:
        Applies \texttt{simplify} and \texttt{expand} to normalize symbolic size expressions such as converting
        $2a(a+1) + a(a^2 + 3a + 1) + 4a$ to $a^3 + 5a^2 + 7a$.
}

\subsubsection{Symbolic disjointness checking using \SymPy\ and \ZT}\label{sec:symbolicdisjoint}

\AC\ constructs a symbolic system of constraints to determine whether two symbolic sets are disjoint.
The key idea is to introduce a symbolic integer variable \( z \) that hypothetically belongs to both sets.
Then, disjointness reduces to checking whether the conjunction of the constraints defining both sets is unsatisfiable.
\texttt{Z3} is used to test satisfiability.

\myitemize{11pt}{
    \item \textbf{Domain constraints (bounds)}:
    For each interval, we require \( z \in [\texttt{lower}, \texttt{upper}] \), encoded via \texttt{Ge(z, lower)} and \texttt{Le(z, upper)}.
    If sets $A$ and $B$ are defined over \( [1, a^2 + 2a] \) and \( [a^2 + 3a + 1, a^3 + 5a^2 + 7a] \), respectively, then
        \[
        z \geq 1 \land z \leq a^2 + 2a \quad \text{and} \quad z \geq a^2 + 3a + 1 \land z \leq a^3 + 5a^2 + 7a .
        \]
        These bounds are incompatible, so $A$ and $B$ are disjoint.

    \item \textbf{Divisibility and non-divisibility constraints}:
    Divisibility constraints are encoded as \( d \mid z \) using \(\texttt{Mod}(z, d) == 0\), 
    and non-divisibility as \( d \nmid z \) using \(\texttt{Mod}(z, d) \neq 0\).  
    Suppose set $A$ contains all \( z \in [1, ab] \) such that \( b \mid z \), 
    while set $B$ contains those where \( b \nmid z \). 
    The constraint system \( z \equiv 0 \mod b \land z \not\equiv 0 \mod b \) 
    is unsatisfiable, so the sets are disjoint.

    \item \textbf{Format expression constraints}:
    To avoid symbol collisions in format expressions, 
    all format variables are renamed using fresh symbols (e.g., \texttt{\_fmt\_i}).
    If a set is defined by a format expression \( z = f(i_1, \dots, i_k) \), 
    then we add \texttt{sp.Eq(z, format\_expr)} along with symbolic bounds on the variables \( i_j \)
    (e.g., \( 0 \leq i < a^2 \)).

    For instance, consider the set \(A\) defined by
      \[
        z = a(a+1)i + aj + k \quad \text{where}\quad 0 \leq i < a^2, \quad 0 \leq j \leq a, \quad \text{and} \quad 2\lfloor j/2 \rfloor + 1 \leq k \leq a - 1,
      \] with constraint \(z\not\equiv{0}\pmod{a}\).  Also, let the set \(B\) be defined by
        \[
        z = t(a^3 + a^2), \quad 1 \leq t \leq \lfloor (a - 1)/2 \rfloor.
        \]
      Now, assume \( z \in A \cap B \). Then $z\in B$ implies \(z\equiv{0}\pmod{a}\), 
      and this contradicts the constraint from \(A\). Therefore, \(A\cap B=\emptyset\).  
  
}

\subsubsection{Monochromatic solution analysis using \ZT}\label{sec:analyzer}
For a linear equation $\sum_{i=1}^{m-1}a_ix_i=a_mx_m$ and $t\geq 1$ sets (each given a colour class), the tool now
generates all possible ways to assign the \(m-1\) left-hand-side (LHS) variables across \(t\) sets using Cartesian products.
	\AC\ iterates over every way to assign the LHS variables to the same colour class
	and for every colour class
  it generates a list of cases for verification.

The prover systematically explores all combinations of sets assigned to the variables in the equation, 
where each variable is taken from the same colour class. For each such case, it verifies 
that no solution exists satisfying the equation under the imposed divisibility, format, and value constraints. 
If all such cases succeed this confirms the absence of monochromatic solutions.

\myitemize{11pt}{
  \item For each assignment of sets to the variables in the equation, 
  the tool aggregates all associated symbolic constraints: interval bounds (e.g., \(x \in [1, a^3]\)), format constraints (e.g., \(x = i a^2 + j a\)), and arithmetic filters (e.g., \(y \nmid a\), \(z \mid a^2\)).

  \item It then encodes the target equation (e.g., \(a x + a y = (a+1) z\)) and all constraints into an SMT expression, along with assumptions (e.g., \(a \geq 7\), \(\texttt{prime}(a)\), \(\gcd(a, b) = 1\)).
  The translation process supports a broad range of symbolic constructs in \SymPy, including integer inequalities, polynomial equalities, floor and ceiling operations, divisibility conditions, and logical connectives.
  
  \item The solver checks that the symbolic system is unsatisfiable.
  The contradiction analysis is exhaustive—it verifies that under all parameter values satisfying the assumptions, the given assignment cannot yield a valid solution.

}
If every such assignment leads to a contradiction, the tool concludes that the equation has no monochromatic solution in the constructed colouring, and thus the lower bound for the Rado number is symbolically proven.

\begin{example}\label{example:z3}
We illustrate a symbolic case analyzed by AutoCase, where each variable in the equation \( a x_1 + a x_2 = (a+1) x_3 \) is assigned to distinct symbolic intervals with bounds, divisibility filters, and format expressions. The analyzer constructs the following constraint system:
\myitemize{11pt}{
  \item \textbf{Global Assumptions}: \( a \in \mathbb{Z} \) and $a\geq1$
  \item \textbf{Equation Constraint}: \( a x_1 + a x_2 = (a + 1) x_3 \)
  \item \textbf{Bounds}: \( x_1, x_2, x_3 \in [1, a^2] \quad \Rightarrow \quad x_i\geq 1, x_i\leq a^2\) for\/ \(i\in\{1,2,3\}\)
  \item \textbf{Divisibility Filter} on \(x_1\): \( a \mid x_1 \quad \Rightarrow \quad x_1 = a \cdot k_1,\quad k_1 > 0 \)
  \item \textbf{Non-divisibility Filter} on \(x_2\): \( a^2 \nmid x_2 \quad \Rightarrow \quad x_2 \neq a^2 \cdot k_2 \; \text{ for all }\; k_2 > 0 \)
  \item \textbf{Format Expression} for \(x_3\): \[ x_3 = a(a+1)i + aj + k \quad \text{with} \quad 0 \leq i < a^2, 0 \leq j \leq a, \text{ and }\; 2\lfloor {j}/{2} \rfloor + 1 \leq k \leq a - 1 \]
  \item \textbf{Total Constraint Set} (passed to \ZT):
  \[
  \begin{aligned}
    & a x_1 + a x_2 = (a+1) x_3 \\
    & x_1 = a \cdot k_1,\quad k_1 > 0 \\
    & \forall{k_2 > 0}, x_2 \neq a^2 \cdot k_2 \\
    & x_3 = a(a+1)i + aj + k \\
    & 0 \leq i < a^2,\quad 0 \leq j \leq a,\quad 2\lfloor {j}/{2} \rfloor + 1 \leq k \leq a - 1 \\
    & x_1, x_2, x_3 \in [1, a^2]
  \end{aligned}
  \]
  }
The set of constraints in \ZT\ are given in Appendix \ref{appendix:z3}\@.
If \ZT\ determines that this system is unsatisfiable for all admissible values of \(a\), then this configuration leads to a contradiction—contributing to the overall proof that the equation has no monochromatic solution in the constructed colouring.
\end{example}
 
\subsection{Proving with \AC: An example}\label{sec:ac-example}
A general technique to prove a lower bound is to construct a colouring avoiding forbidden patterns. Specifically, to show $R_k({\cal E})\geq N+1$, constructing a $k$-colouring of the
positive integers $[1,N]$ avoiding any monochromatic solution to equation ${\cal E}$ would be sufficient.
As an example, Proposition~\ref{prop:r3a11} states $R_3({\cal E}(3, 0; a, 1, 1))=a^3+5a^2+7a+1$.
Although we already showed this proposition is a consequence of a theorem of
Chang, De Loera, and Wesley~\cite{chang2022}, as an example we now show how a form of Proposition~\ref{prop:r3a11}
(with the equality replaced by a lower bound) can be proven using \AC.

\subsubsection{Construction of the partition}
Based on empirical observations of solutions to small instances of the problem, using certificates (satisfying assignments) generated by the SAT solver CaDiCaL, we identify a partition of $[1, a^3+5a^2+7a]$ into seven intervals $P_0$, $\dotsc$, $P_6$ along with a corresponding colouring function~$\delta$.

Consider the partition of $[1, a^3+5a^2+7a]$ using the intervals
\begin{align*}
P_0 &= [1,a], \\
P_1 &= [a+1, a^2+2a], \\
P_2 &= [a^2+2a+1, a^2+3a], \\
P_3 &= [a^2+3a+1, a^3+4a^2+4a], \\
P_4 &= [a^3+4a^2+4a+1, a^3+4a^2+5a], \\
P_5 &= [a^3+4a^2+5a+1, a^3+5a^2+6a], \\
P_6 &= [a^3+5a^2+6a+1, a^3+5a^2+7a].
\end{align*}
From these, we define the colouring function $\delta\colon [1, a^3+5a^2+7a] \rightarrow [0,1,2]$ as
\[
\delta(i) = 
	\begin{cases}
		0 & \text{if $i\in P_0 \cup P_2 \cup P_4 \cup P_6$,} \\
		1 & \text{if $i\in P_1 \cup P_5$,} \\
		2 & \text{if $i\in P_3$.} \\
	\end{cases}
\]

\subsubsection{Input processing}
The prover considers all possible pairs $x$, $y$ drawn from the same colour class, and verifies that \(z=ax+y\) 
cannot belong to any set in the same colour class.
Based on the symbolic colouring provided by $\delta$, the prover enumerates all ways of assigning $x$, $y$, and $z$ to the same colour. Specifically, it generates $4^3 = 64$ cases for colour $0$, $2^3 = 8$ cases for colour $1$, and $1^3=1$ case for colour $2$, yielding a total of $73$ cases. These cases correspond to the combinations of source intervals for $(x, y)$ and target intervals for $z$.

\subsubsection{Monochromaticity analysis}
\begin{example}\label{example:colour-0-constraints}
We check the cases where \( x \in P_0 \) and \( y \in P_0 \).
The computed \( z = ax + y \) must not belong to any of \( P_0, P_2, P_4, P_6 \).
The reasons why \( z \) cannot belong to any of \( P_0, P_2, P_4, P_6 \) are as follows:

\myitemize{11pt}{
\item Lower bound of \( z \):  Since \( x, y \in P_0 = [1, a] \), their smallest possible values are \( x = 1 \) and \( y = 1 \). This gives \( z = ax + y \geq a(1) + 1 = a + 1 \).
\item Upper bound of \( z \):  The largest values in \( P_0 \) are \( x = a \) and \( y = a \). This gives \( z = ax + y \leq a(a) + a = a^2 + a \).
\item Now, we check whether \( z \) can belong to any forbidden partition:
	\myitemize{11pt}{
		\item \( P_0 = [1, a] \): Since \( z \geq a+1 \), \( z \notin P_0 \).
		\item \( P_2 = [a^2 + 2a + 1, a^2 + 3a] \):  The maximum possible \( z \) is \( a^2 + a \), which is less than \( \min(P_2) = a^2 + 2a + 1 \).  Thus, \( z \notin P_2 \).
		\item \( P_4 = [a^3 + 4a^2 + 4a + 1, a^3 + 4a^2 + 5a] \):  Clearly, \( z \leq a^2 + a \) is much smaller than \( \min(P_4) \), so \( z \notin P_4 \).
		\item \( P_6 = [a^3 + 5a^2 + 6a + 1, a^3 + 5a^2 + 7a] \):  Again, \( z \) is far smaller than \( \min(P_6) \), so \( z \notin P_6 \).

	}
}
Since \(z\) is forced outside the monochromatic partitions, we conclude that no monochromatic solution exists in colour 0 when \( x \in P_0 \) and \( y \in P_0 \).
\end{example} 
The same type of analysis applies to all other colour cases, ensuring that the generated constraints exhaustively eliminate any monochromatic solution.
As a result, \AC\ has shown that $R_3({\cal E}(3, 0; a, 1, 1))\geq a^3+5a^2+7a+1$ for all $a\geq1$.

\subsection{Proof of Theorem \ref{th:lb-abb} (lower bound for $R_3({\cal E}(3, 0; a, b, b)))$}\label{sec:thm1proof}

\subsubsection{Construction of the partition}
  Since $\gcd(a,b)=1$, the variable $z$ to be an integer in the equation $ax+by=bz$, the variable $x$ must be a multiple of $b$.
  Consider the partition of $[1, a^3+a^2+(2b+1)a]$ using the intervals
  \begin{align*}
  P_0 &= [1,ba], \\
  P_1 &= [ba+1, b^2a], \\
  P_2 &= [b^2a+1, ba^2+ba], \\
  P_3 &= [ba^2+ba+1, a^3+a^2+(b+1)a], \text{ and} \\
  P_4 &= [a^3+a^2+(b+1)a+1, a^3+a^2+(2b+1)a].
  \end{align*}
  Note that indeed $[1,a^3+a^2+(2b+1)a]=P_0\,\cup\,P_1\,\cup\,P_2\,\cup\,P_3\,\cup\,P_4$ and
  $P_i\cap P_j=\emptyset$ for $i\neq j$. Define the following sets to
  filter the above intervals for colouring in Dark Gray (0), Red (1), and Blue~(2):
  \begin{align*}
  R_1 &= \set{v \in P_0: v\not\equiv 0\pmod{b}} \\
  R_2 &= \set{v \in P_0\cup P_1: v\equiv 0\pmod{b^2}}\\
  R_3 &= P_4 \\
  B_1 &= \set{v \in P_0\cup P_1 : \begin{array}{l}
              v\equiv 0\pmod{b}, \\ 
                                                          v\not\equiv 0\pmod{b^2}
            \end{array}}\\
  B_2 &= \set{v \in P_2: v\equiv 0\pmod{b}}\\
  D_1 &= \set{v \in P_1\cup P_2: v\not\equiv 0\pmod{b}}\\
  D_2 &= P_3
  \end{align*}
  Now, consider the colouring $\delta\colon [1, a^3+a^2+(2b+1)a] \rightarrow [0,1,2]$ defined by
  \begin{align*}
  \delta(i) &= 
    \begin{cases}
    0 & \text{if $i\in D_1 \cup D_2$,} \\ 
    1 & \text{if $i\in R_1 \cup R_2 \cup R_3$,} \\
    2 & \text{if $i\in B_1 \cup B_2$.} \\
    \end{cases}
  \end{align*}
  Example~\ref{example:abb} provides a visual representation of this colouring.
  \begin{example}\label{example:abb}
    Below is the colouring of $[1,108]$ used to prove that $R_3({\cal E}(3,0;4,3,3))\geq 109$.
    \begin{center}
    \small
    \begin{tabular}{ccccc}
    \begin{tabular}{|c|c|c|}
    \hline
    \color{red} r1 & \color{red} r2 & \color{blue} b3 \\ \hline
    \color{red} r4 & \color{red} r5 & \color{blue} b6 \\ \hline
    \color{red} r7 & \color{red} r8 & \color{red} r9 \\ \hline
    \color{red} r10 & \color{red} r11 & \color{blue} b12 \\ \hline
    \multicolumn{3}{|c|}{$P_0$} \\ \hline
    \end{tabular}
    &
    \begin{tabular}{|c|c|c|}
    \hline
    13 & 14 & \color{blue} b15 \\ \hline
    16 & 17 & \color{red} r18 \\ \hline
    19 & 20 & \color{blue} b21 \\ \hline
    22 & 23 & \color{blue} b24 \\ \hline
    25 & 26 & \color{red} r27 \\ \hline
    28 & 29 & \color{blue} b30 \\ \hline
    31 & 32 & \color{blue} b33 \\ \hline
    34 & 35 & \color{red} r36 \\ \hline
    \multicolumn{3}{|c|}{$P_1$} \\ \hline
    \end{tabular}
    &
    \begin{tabular}{|c|c|c|}
    \hline
    37 & 38 & \color{blue} b39 \\ \hline
    40 & 41 & \color{blue} b42 \\ \hline
    43 & 44 & \color{blue} b45 \\ \hline
    46 & 47 & \color{blue} b48 \\ \hline
    49 & 50 & \color{blue} b51 \\ \hline
    52 & 53 & \color{blue} b54 \\ \hline
    55 & 56 & \color{blue} b57 \\ \hline
    58 & 59 & \color{blue} b60 \\ \hline
    \multicolumn{3}{|c|}{$P_2$} \\ \hline
    \end{tabular}

    &
    \begin{tabular}{|c|c|c|}
    \hline
    61 & 62 & 63 \\ \hline
    64 & 65 & 66 \\ \hline
    67 & 68 & 69 \\ \hline
    70 & 71 & 72 \\ \hline
    73 & 74 & 75 \\ \hline
    76 & 77 & 78 \\ \hline
    79 & 80 & 81 \\ \hline
    82 & 83 & 84 \\ \hline
    85 & 86 & 87 \\ \hline
    88 & 89 & 90 \\ \hline
    91 & 92 & 93 \\ \hline
    94 & 95 & 96 \\ \hline
    \multicolumn{3}{|c|}{$P_3$} \\ \hline
    \end{tabular}
    & 
    \begin{tabular}{|c|c|c|}
    \hline
    \color{red} r97 & \color{red} r98 & \color{red} r99 \\ \hline
    \color{red} r100 & \color{red} r101 & \color{red} r102 \\ \hline
    \color{red} r103 & \color{red} r104 & \color{red} r105 \\ \hline
    \color{red} r106 & \color{red} r107 & \color{red} r108 \\ \hline
    \multicolumn{3}{|c|}{$P_4$} \\ \hline
    \end{tabular}
    \end{tabular}
    \normalsize
    \end{center}

    \end{example}

\begin{proof}[Proof of Theorem \ref{th:lb-abb}.]
  A written proof of Theorem~\ref{th:lb-abb} is provided in Appendix \ref{pr:lb-abb}. The symbolic correctness of the partition and the verification of the cases can be found in the \AC\ repository.
\end{proof}
  
  \subsection{Proof of Theorem \ref{th:lb-aaap1} (lower bound for $R_3({\cal E}(3, 0; a, a, a+1))$)}\label{sec:thm2proof}
  
  \begin{observation}\label{obs:example-grids}
  For odd positive integers $a\geq 7$, 
  there exists a certificate (computed using a SAT solver) for ${R_3}({\cal E}(3, 0; a, a, a+1)) \geq a^3(a+1)$ of length $a^3(a+1)-1$
  where certain repetitive patterns appear as blocks of coloured integers of size $a(a+1)$.
  Consider the following examples for ${R_3}({\cal E}(3, 0; 7, 7, 8))$.
  \begin{center}
  \scriptsize
  \begin{tabular}{cc}
  \begin{tabular}{|c|c|c|c|c|c|c|}
  \hline
  \color{red} r1& \color{red} r2 & \color{red} r3 & \color{red} r4 & \color{red} r5 & \color{red} r6 & 7 \\ \hline
  \color{red} r8& \color{red} r9 & \color{red} r10 & \color{red} r11 & \color{red} r12 & \color{red} r13 & 14 \\ \hline
  \color{blue} b15 & \color{blue} b16 & \color{red} r17 & \color{red} r18 & \color{red} r19 & \color{red} r20 & 21 \\ \hline
  \color{blue} b22 & \color{blue} b23 & \color{red} r24 & \color{red} r25 & \color{red} r26 & \color{red} r27 & 28 \\ \hline
  \color{blue} b29 & \color{blue} b30 & \color{blue} b31 & \color{blue} b32 & \color{red} r33 & \color{red} r34 & 35 \\ \hline
  \color{blue} b36 & \color{blue} b37 & \color{blue} b38 & \color{blue} b39 & \color{red} r40 & \color{red} r41 & 42 \\ \hline
  \color{blue} b43 & \color{blue} b44 & \color{blue} b45 & \color{blue} b46 & \color{blue} b47 & \color{blue} b48 & \color{red} r49 \\ \hline
  \color{blue} b50 & \color{blue} b51 & \color{blue} b52 & \color{blue} b53 & \color{blue} b54 & \color{blue} b55 & 56 \\ \hline
  \multicolumn{7}{|c|}{Colour block 1 for $R_3({\cal E}(3,0; 7,7,8))$} \\ \hline
  \end{tabular}
  &
  \begin{tabular}{|c|c|c|c|c|c|c|}
  \hline
  \color{red} r57 & \color{red} r58 & \color{red} r59 & \color{red} r60 & \color{red} r61 & \color{red} r62 & 63 \\ \hline
  
  \color{red} r64 & \color{red} r65 & \color{red} r66 & \color{red} r67 & \color{red} r68 & \color{red} r69 & 70 \\ \hline
  
  \color{blue} b71 & \color{blue} b72 & \color{red} r73 & \color{red} r74 & \color{red} r75 & \color{red} r76 & 77\\ \hline
  
  \color{blue} b78 & \color{blue} b79 & \color{red} r80 & \color{red} r81 & \color{red} r82 & \color{red} r83 & 84\\ \hline
  
  \color{blue} b85 & \color{blue} b86 & \color{blue} b87 & \color{blue} b88 & \color{red} r89 & \color{red} r90 & 91\\ \hline
  
  \color{blue} b92 & \color{blue} b93 & \color{blue} b94 & \color{blue} b95 & \color{red} r96 & \color{red} r97 & \color{red} r98\\ \hline
  
  \color{blue} b99 & \color{blue} b100 & \color{blue} b101 & \color{blue} b102 & \color{blue} b103 & \color{blue} b104 & 105\\ \hline
  
  \color{blue} b106 & \color{blue} b107 & \color{blue} b108 & \color{blue} b109 & \color{blue} b110 & \color{blue} b111 & 112 \\ \hline
  \multicolumn{7}{|c|}{Colour-block 2 for $R_3({\cal E}(3,0; 7,7,8))$} \\ \hline
  \end{tabular}
  \\
  \\
  \begin{tabular}{|c|c|c|c|c|c|c|}
    \hline
    \color{red} r337 & \color{red} r338 & \color{red} r339 & \color{red} r340 & \color{red} r341 & \color{red} r342 & \color{blue} b343 \\ \hline
    
    \color{red} r344 & \color{red} r345 & \color{red} r346 & \color{red} r347 & \color{red} r348 & \color{red} r349 & 350 \\ \hline
    
    \color{blue} b351 & \color{blue} b352 & \color{red} r353 & \color{red} r354 & \color{red} r355 & \color{red} r356 & 357\\ \hline
    
    \color{blue} b358 & \color{blue} b359 & \color{red} r360 & \color{red} r361 & \color{red} r362 & \color{red} r363 & 364\\ \hline
    
    \color{blue} b365 & \color{blue} b366 & \color{blue} b367 & \color{blue} b368 & \color{red} r369 & \color{red} r370 & 371\\ \hline
    
    \color{blue} b372 & \color{blue} b373 & \color{blue} b374 & \color{blue} b375 & \color{red} r376 & \color{red} r377 & 378\\ \hline
    
    \color{blue} b379 & \color{blue} b380 & \color{blue} b381 & \color{blue} b382 & \color{blue} b383 & \color{blue} b384 & 385\\ \hline
    
    \color{blue} b386 & \color{blue} b387 & \color{blue} b388 & \color{blue} b389 & \color{blue} b390 & \color{blue} b391 & \color{red} r392 \\ \hline
    \multicolumn{7}{|c|}{Colour-block 7 for $R_3({\cal E}(3,0; 7,7,8))$} \\ \hline
    \end{tabular}
  & 
  \begin{tabular}{|c|c|c|c|c|c|c|}
    \hline
    \color{red} r2689 & \color{red} r2690 & \color{red} r2691 & \color{red} r2692 & \color{red} r2693 & \color{red} r2694 & \color{blue} b2695 \\ \hline
    
    \color{red} r2696 & \color{red} r2697 & \color{red} r2698 & \color{red} r2699 & \color{red} r2700 & \color{red} r2701 & 2702 \\ \hline
    
    \color{blue} b2703 & \color{blue} b2704 & \color{red} r2705 & \color{red} r2706 & \color{red} r2707 & \color{red} r2708 & 2709\\ \hline
    
    \color{blue} b2710 & \color{blue} b2711 & \color{red} r2712 & \color{red} r2713 & \color{red} r2714 & \color{red} r2715 & 2716\\ \hline
    
    \color{blue} b2717 & \color{blue} b2718 & \color{blue} b2719 & \color{blue} b2720 & \color{red} r2721 & \color{red} r2722 & 2723\\ \hline
    
    \color{blue} b2724 & \color{blue} b2725 & \color{blue} b2726 & \color{blue} b2727 & \color{red} r2728 & \color{red} r2729 & 2730\\ \hline
    
    \color{blue} b2731 & \color{blue} b2732 & \color{blue} b2733 & \color{blue} b2734 & \color{blue} b2735 & \color{blue} b2736 & 2737\\ \hline
    
    \color{blue} b2738 & \color{blue} b2739 & \color{blue} b2740 & \color{blue} b2741 & \color{blue} b2742 & \color{blue} b2743 &  \\ \hline
    \multicolumn{7}{|c|}{Colour-block 49 for $R_3({\cal E}(3,0; 7,7,8))$} \\ \hline
    \end{tabular}
  
  \end{tabular}
  \normalsize
  \end{center}
  The integers that are not divisible by\/ $a$ get coloured in a consistent way in each block, while
  the integers that are divisible by\/ $a$ need to be carefully coloured in such a way (a different way in each block)
  to avoid monochromatic solutions to the equation\/ $ax+ay=(a+1)z$.
  \end{observation}
  
  \subsubsection{Construction of the partition}
  Before we prove the lower bound, we construct a partition of the set $[1,N]$ 
  where $N = a^3(a+1)-1$ with $a\geq 7$ being an odd positive integer. 
  Consider, 
  \[[1,N+1] = \set{i\cdot a(a+1)+a\cdot j+k: 0\leq i\leq a^2-1, 0\leq j\leq a, 1\leq k \leq a}.\]
Let \(S_{d}(N)\) and \(\overline{S_d}(N)\) denote \(\set{x \in [1,N]: d \mid x}\) and \(\set{x \in [1,N]: d \nmid x}\), respectively.
We have the partition $[1,N] = S_a(N)\cup \overline{S_a}(N)$ with
\begin{align*}
S_a(N) &= \set{i\cdot a(a+1)+a\cdot j+a: 0\leq i\leq a^2-1, 0\leq j\leq a} \setminus \set{N+1}\text{, and} \\
\overline{S_a}(N) &= \set{i\cdot a(a+1)+a\cdot j+k: 0\leq i\leq a^2-1, 0\leq j\leq a, 1\leq k\leq a-1}.
\end{align*}
Let \(\overline{S_a}(N) = R_{\ell} \cup B_{\ell}\) be a partition defined by
\begin{align*}
R_{\ell} &= \set{i\cdot a(a+1)+a\cdot j+k: 0\leq i\leq a^2-1, 0\leq j\leq a, 2\lfloor{j/2}\rfloor+1\leq k\leq a-1}\text{, and}\\
B_{\ell} &= \set{i\cdot a(a+1)+a\cdot j+k: 0\leq i\leq a^2-1, 0\leq j\leq a, 1\leq k\leq 2\lfloor{j/2}\rfloor}.
\end{align*}
Now, consider the set
\[S_{a^2}(N) = \set{i\cdot a(a+1)+a\cdot j+a: 0\leq i\leq a^2-1, 0\leq j\leq a, a \mid (i+j+1)} \setminus \set{N+1} . \]
Let \(S_{a^2}(N) = R_{r} \cup B_{r}\) be a partition defined by
\begin{align*}
B_r &= \bigcup_{i=1}^{a-1}\set{a^4+i a^2} \cup  
\bigcup_{i=1}^{a-1}\set{i a^3+j a^2: 0 \leq j\leq i-1} \cup
\bigcup_{i=(a+1)/2}^{a-1}\set{i a^3+i a^2}\text{, and} \\
R_r &= \set{a^4} \cup 
\bigcup_{i=0}^{a-1}\set{i a^3+j a^2: i+1 \leq j\leq a-1} \cup
\bigcup_{i=1}^{(a-1)/2}\set{i a^3+i a^2} .
\end{align*}
Our colouring $\delta\colon [1, N] \rightarrow [0,1,2]$ will be defined by
\[
  \delta(x) = 
    \begin{cases}
    2 & \text{if $x\in B_{\ell} \cup B_r$,}\\
    1 & \text{if $x\in R_{\ell} \cup R_r$,}\\
    0 & \text{if $x\in S_{a}(N) \setminus S_{a^2}(N)$.}\\
    \end{cases}
\]

  \begin{proof}[Proof of Theorem \ref{th:lb-aaap1}.]
    A written proof of Theorem~\ref{th:lb-aaap1} is provided in Appendix \ref{pr:lb-aaap1}.
    The symbolic correctness of the partition and the verification of the cases can be found in the \AC\ repository.
  \end{proof}

\section{Conclusion}
  
In this work, we explored the computation of 3-colour Rado numbers for certain 3-term linear equations. Our SAT-based approach enabled us to extend the known values of \( R_3({\cal E}(3,0;a,b,b)) \) and \( R_3({\cal E}(3,0;a,a,b)) \), computing previously unknown Rado numbers and formulating new conjectures based on observed patterns. 
We proved the lower bounds \( R_3({\cal E}(3,0;a,b,b)) \geq a^3 + a^2 + (2b+1)a + 1 \) and \( R_3({\cal E}(3,0;a,a,a+1)) \geq a^3(a+1) \) for certain ranges of $a$ and $b$. The case-based proofs were automatically verified using a combination of symbolic computation and SMT solvers in a tool we call \AC.

We did not attempt to compute Rado numbers involving more than three colours due to the rapid growth in the SAT instances, and for simplicity we focused on linear equations with a constant term of zero.  In principle, our methods should be applicable to nonhomogeneous linear equations and to Rado numbers involving four or more colours, though they may require more computational resources or more advanced combinatorial optimization techniques.
We leave such cases for future work.

\bibliography{references}

\appendix
\section{Proof of Lemma~\ref{lem:r1value}}\label{sec:r1proof}

In this section we assume $a$ and $b$ are positive coprime integers.  Lemma~\ref{lem:r1value} then says
\begin{equation} R_1({\cal E}(3,0;a,b,b))=\max(a+1,b) \quad \text{and} \quad R_1({\cal E}(3,0;a,a,b))=\begin{cases}2a & \text{if $b=1$,} \\
\max(a,\lceil b/2\rceil) & \text{if $b>1$.}\end{cases} \label{eq:r1value} \end{equation}

For the first equation of~\eqref{eq:r1value}, we want to find the solution $(x,y,z)$ of $ax+by=bz$ that minimizes $\max(x,y,z)$;
then $R_1({\cal E}(3,0;a,b,b))=\max(x,y,z)$.
Since $a$ and $b$ are coprime, $x$ must be divisible by $b$, so $(b,1,a+1)$ is the solution of $ax+by=bz$
with smallest possible $x$ and $y$ values.  Since $z=(ax+by)/b$, the solution with the smallest possible $x$ and $y$
will also minimize $z$.  Thus, the smallest possible value of $\max(x,y,z)$ is $\max(b,1,a+1)=\max(a+1,b)$.

For the second equation of~\eqref{eq:r1value}, we want to find the solution $(x,y,z)$ of $ax+ay=bz$ that minimizes $\max(x,y,z)$;
then $R_1({\cal E}(3,0;a,a,b))=\max(x,y,z)$.
Since $a$ and $b$ are coprime, $z$ must be divisible by~$a$, so say $z=ak$, giving $x+y=bk$.
Then $\max(x,y,z)$ is minimized when $x$ and $y$ are as close to equal as possible (i.e., $x\approx y\approx bk/2$).

If $b$ is even, $\max(x,y,z)$ is minimized when $k=1$ and $x=y=b/2$, giving $\max(x,y,z)=\max(a,b/2)$,
so $R_1({\cal E}(3,0;a,a,b))=\max(a,b/2)$ when $b$ is even.
If $b>1$ is odd, then $\max(x,y,z)$ is minimized when $k=1$ and $\{x,y\}=\{\lfloor b/2\rfloor,\lceil b/2\rceil\}$,
so $R_1({\cal E}(3,0;a,a,b))=\max(a,\lceil b/2\rceil)$ when $b>1$ is odd.
If $b=1$, then $\max(x,y,z)$ is minimized when $k=2$ and $x=y=1$ (note $x+y=k$ has no solutions in positive integers
when $k=1$) and then $R_1({\cal E}(3,0;a,a,b))=\max(1,1,2a)=2a$.

\section{Written proof of Theorem \ref{th:lb-abb}}\label{pr:lb-abb}
Suppose there is a solution to $ax+by=bz$ monochromatic in colour $0$ (Dark Gray).
Since \(\text{gcd}(a,b)=1\) and there is no multiple of $b$ in $D_1$, the variable $x$ must be in $D_2$. Taking $x=ba^2+ba+b$ and $y=\min(D_1)=ba+1$ as smallest possible values, we get $z = a^3+a^2+(b+1)a+1$, but $\delta(z)=1$ (Red) for $z \geq a^3+a^2+(b+1)a+1$, a contradiction.

Suppose there is a solution to $ax+by=bz$ monochromatic in colour $1$ (Red). 
Since $R_1$ does not contain a multiple of $b$, variable $x$ cannot be in $R_1$.  
Also, if $x, y \in R_3$, then clearly $z > a^3+a^2+(2b+1)a$. So, $x$ and $y$ both cannot be in $R_3$. Now we consider the following cases:
\begin{itemize}
  \item [(a)] $x\in R_2$ and $y\in R_1$:  
We have $z={ax\over b}+y = baw+y$ for $1\leq w\leq a$. 
Since, $y \not\equiv{0}\pmod{b}$, we have $baw+y \not\equiv{0}\pmod{b}$ which implies
$baw+y \not\equiv{0}\pmod{b^2}$. Hence $baw+y \not\in R_2$. If $baw+y \in R_1$, then 
since $w\geq 1$ and $y\geq 1$, we have $baw+y > ba = \max{(R_1)}$, a contradiction.

  \item [(b)] $x\in R_2$ and $y\in R_2$: We have $z \not\in R_3$ since 
\begin{align*}
z &\leq a(b^2a)/b+b^2a = ba^2+b^2a \\ 
  &\leq a^3+a^2+(b+1)a < \min(R_3)
\end{align*} 
when $a^2+a+b > b^2 + ba$. Suppose $z\in R_2$. In that case,
\(ax_1+bx_2=bx_3\) with \(1 \leq x_1,x_2,x_3 \leq ab^2\) implies \(ak_1+bk_2=bk_3\) 
with \(1 \leq k_1, k_2, k_3 \leq a\). Then, $k_1 = (b/a)(k_3 - k_2)$ has no integer solution since 
\(\text{gcd}(a,b)=1\) and \(0\leq k_3-k_2<a\). Hence, $\delta(z)\neq 1$.
  \item [(c)] $x\in R_3$ and $y\in R_1$:  
Consider the condition, with $a^3+a^2+(b+1)a + c$ being the smallest integer in $R_3$ which is divisible by $b$. Then, we have $a^2+a+b > ba+b^2$ which can be re-written as
${a\over b}(a^3+a^2+(b+1)a + c)+1 > {a\over b}(ba^2 + ab^2 + a + c)+1$, which implies \[z > (a^3 + a^2b+a^2/b + ac/b) > \max(R_3).\]
Hence, $\delta(z)\neq 1$, a contradiction.
  \item [(d)] $x\in R_3$ and $y\in R_2$:
As in the previous case, with $a^3+a^2+(b+1)a + c$ being the smallest integer in $R_3$ which is divisible by $b$, and $b^2$ being the smallest integer in $R_2$, we get $z > \max(R_3)$ implying 
\(z\not\in R_3\) implying $\delta(z)\neq 1$, a contradiction. 
\end{itemize}

Suppose there is a solution to $ax+by=bz$ monochromatic in colour $2$ (Blue). Then we have the following cases:
\begin{itemize}
\item [(i)] If $x, y \in B_2$, then 
\begin{align*}
z&={ax\over b}+y\geq {a(b^2a+b)\over b}+(b^2a+b) \\ 
  &= ba^2+b^2a+a+b > \max(B_2),
\end{align*}
which results a contradiction since integers beyond $\max(B_2)$ are coloured in either colour $0$ or in colour $1$.

\item [(ii)] If $x, y \in B_1$, then in the equation $z={ax\over b}+y$, we have ${x\over b} \equiv 1,2,\ldots,b-1\pmod{b}$. Since $\gcd(a,b)=1$, we have ${ax\over b}\not\equiv{0}\pmod{b}$. Hence, $z={ax\over b}+y\not\equiv{b,2b,\ldots,b(b-1)}\pmod{b^2}$ and $z\not\equiv{0}\pmod{b}$ implying $\delta(z)\neq 2$, a contradiction.

\item [(iii)] If $x\in B_1$ and $y\in B_2$, then as in the previous case, we have ${ax\over b}\not\equiv{0}\pmod{b}$. Since $y\equiv{0}\pmod{b}$, we have ${ax\over b}+y\not\equiv{0}\pmod{b}$, that is, $z\not\in B_2$. 
Hence, $\delta(z)\neq 2$, a contradiction.

\item [(iv)] If $x\in B_2$ and $y\in B_1$, then
$z = {ax\over b}+y \geq {a(b^2a+b)/b+b} = a^2b+a+b > \max(B_1)$.
Also, since $z\not\equiv{0}\pmod{b}$, we have $z\not\in B_2$. Hence, $\delta(z)\neq 2$, a contradiction.
\end{itemize}
Hence, $$R_3({\cal E}(3, 0; a, b, b)) \geq a^3+a^2+(2b+1)a+1$$
if $a^2+a+b > b^2 + ba$. 

\section{Written proof of Theorem \ref{th:lb-aaap1}}\label{pr:lb-aaap1}
\subsection{Integrality properties}
  
\begin{lemma}\label{integrality}
  Let\/ $a$ be a positive integer. There is no integer solution $x_1,x_2,x_3$ to the equation \(ax_1 + ax_2 = (a+1)x_3\) 
  if any of the following is true:
  \myitemize{11pt}{
    \item [a.] \(x_1,x_2,x_3\not\equiv{0}\pmod{a}\);
    \item [b.] \(x_1,x_2,x_3\equiv{0}\pmod{a}\) and \(x_1,x_2,x_3\not\equiv{0}\pmod{a^2}\);
    \item [c.] \(x_1,x_2\equiv{0}\pmod{a^2}\) and \(x_3\not\equiv{0}\pmod{a}\);
    \item [d.] \(x_1,x_3\equiv{0}\pmod{a^2}\) and \(x_2\not\equiv{0}\pmod{a}\);
    \item [e.] \(x_2,x_3\equiv{0}\pmod{a^2}\) and \(x_1\not\equiv{0}\pmod{a}\);
    \item [f.] \(x_1\equiv{0}\pmod{a^2}\) and \(x_2,x_3\not\equiv{0}\pmod{a}\);
    \item [g.] \(x_2\equiv{0}\pmod{a^2}\) and \(x_1,x_3\not\equiv{0}\pmod{a}\);
  }

  \begin{proof}
    In each of the cases, assume that an integer solution \(x_1,x_2,x_3\) exists to \(ax_1+ax_2=(a+1)x_3\).
    \myitemize{11pt}{
      \item [a.] Simplifying the equation, we get \(x_1+x_2=x_3+x_3/a\), which due to the integrality assumption, implies \(a|x_3\), that is, \(x_3\equiv{0}\pmod{a}\), a contradiction.
      \item [b.] For \(i=1,2,3\), since \(x_i\equiv{0}\pmod{a}\) and \(x_i\not\equiv{0}\pmod{a^2}\), let \(x_i=a\cdot k_i\) such that \(k_i\not\equiv{0}\pmod{a}\). Upon substitution,
      we obtain \(a^2\cdot(k_1+k_2-k_3)=a\cdot k_3\) which implies \(k_3\equiv{0}\pmod{a}\), a contradiction.
      \item [c.] Let \(x_1=a^2k_1\) and \(x_2=a^2k_2\). Then after substitution and simplification, we obtain \(a^2k_1+a^2k_2=x_3+x_3/a\), which due to the integrality assumption, implies \(a|x_3\), that is, \(x_3\equiv{0}\pmod{a}\), a contradiction.
      \item [d.] Let \(x_1=a^2k_1\) and \(x_3=a^2k_3\). Then after substitution and simplification, we obtain \(ak_1+x_2/a=ak_3+k_3\), which due to the integrality assumption, implies \(a|x_2\), that is, \(x_2\equiv{0}\pmod{a}\), a contradiction.
      \item [e.] Let \(x_2=a^2k_2\) and \(x_3=a^2k_3\). Then after substitution and simplification, we obtain \(x_1/a+ak_2=ak_3+k_3\), which due to the integrality assumption, implies \(a|x_1\), that is, \(x_1\equiv{0}\pmod{a}\), a contradiction.
      \item [f.] Let \(x_1=a^2k_1\). Then after substitution and simplification, we obtain \(a^2k_1+x_2=x_3+x_3/a\), which due to the integrality assumption, implies \(a|x_3\), that is, \(x_3\equiv{0}\pmod{a}\), a contradiction.
      \item [g.] Let \(x_2=a^2k_2\). Then after substitution and simplification, we obtain \(x_1+a^2k_2=x_3+x_3/a\), which due to the integrality assumption, implies \(a|x_3\), that is, \(x_3\equiv{0}\pmod{a}\), a contradiction. \qedhere
      }
  \end{proof}
\end{lemma}

\subsection{Proof of good colouring}
\begin{lemma}
  For odd positive integer \(a \geq 7\), there exists no monochromatic solution to \(ax_1+ax_2=(a+1)x_3\) in colour \(0\).
  \begin{proof}
    If \(x \in S_a(N)\setminus \overline{S_{a^2}}(N)\), then by Lemma \ref{integrality}(b), there exists no solution to \(ax_1+ax_2=(a+1)x_3\).
    Hence there is no solution to \(ax_1+ax_2=(a+1)x_3\) monochromatic in colour \(0\).
  \end{proof}
\end{lemma}

\begin{lemma}\label{lem:R_r}
  There exists no monochromatic solution to \(ax_1+ax_2=(a+1)x_3\) if \(x_1,x_2,x_3 \in R_r\).
  \begin{proof} 
    Given \(x_1,x_2\in R_r\), assume that \(x_3\in R_r\) where \(x_3=(ax_1+ax_2) / (a+1)\). Let
    \(x_1=k_1a^2\) and \(x_2=k_2a^2\) for positive integers \(k_1\) and \(k_2\). 
    Substituting for \(x_3\), we get the numerator to be \(a^3(k_1+k_2)\). Since \(a^3\) 
    shares no factor with \(a+1\), if \(a+1\) does not divide \(k_1+k_2\), then $x_3\not\in {\mathbb Z}$, 
    a contradiction. Assume that \(k_1+k_2=m(a+1)\) for some positive integer \(m\). 

    Assume that \(k_1+k_2=m(a+1)\) for some positive integer \(m\). Then, \(z=a^3m\), but for \(1\leq m \leq a-1\), by definition of the sets $B_r$ and $R_r$, we have \(a^3m\in B_r\) and since \(B_r\cap R_r=\emptyset\), we have \(a^3m\not\in R_r\). 
    If \(z=a^4\), then \(m=a\) implying \(k_1+k_2=a(a+1)\) which is impossible since \(k_1+k_2 \leq ai+j \leq a^2-1 < a(a+1)\) given \(0\leq i\leq a-1, i+1\leq j\leq a-1\). 
  \end{proof}
\end{lemma}

\begin{lemma}\label{lem:12ell-3r}
  For odd positive integer \(a \geq 7\), there exists no monochromatic solution to \(ax_1+ax_2=(a+1)x_3\) if \(x_1,x_2 \in R_{\ell}\) and \(x_3\in R_r\).
\begin{proof}
  Since \(x_1,x_2\in R_{\ell}\), let \(x_1 = i_1a(a+1)+r_1\) and \(x_2 = i_2a(a+1)+r_2\) with 
  \(0\leq i_1,i_2\leq a^2-1\), \(1\leq r_1,r_2 < a(a+1)\). Also, let \(x_3= a^2k_3\).
  Therefore, \[k_3 = (i_1+i_2) + {r_1+r_2\over a(a+1)}.\]

  If \(r_1 + r_2 = a(a + 1)\), then by construction of \(B_{\ell}\) and \(R_{\ell}\), the integers \(x_1\) and
  \(x_2\) both not in the same set, a contradiction.
\end{proof}
\end{lemma}

\begin{lemma}
  For odd positive integer \(a \geq 7\), there exists no monochromatic solution to \(ax_1+ax_2=(a+1)x_3\) in colour \(1\).
  \begin{proof}
    Let \(R_r = R_1\cup R_2\cup R_3\) with \(R_1=\set{a^4}\), \(R_2=\bigcup_{i=1}^{a-1}\set{ia^3+ja^2: i+1\leq j\leq a-1}\), and \(R_3=\bigcup_{i=1}^{(a-1)/2}\set{ia^3+ia^2}\).
    There are \(4^3=64\) ways to select an integer for the variables \(x_1,x_2,x_3\) from the \(4\) sets \(R_{\ell}, R_1, R_2, R_3\). 
    Most of these cases can be analyzed directly using Lemma \ref{integrality} as follows:
    \myitemize{11pt}{
      \item \(x_1,x_2,x_3\in R_{\ell}\): No monochromatic solution by Lemma \ref{integrality} (a) covering \(1\) case.
      \item \(x_1\in R_{\ell}, x_2,x_3\in R_r\): No monochromatic solution by Lemma \ref{integrality} (e) covering \(9\) cases.
      \item \(x_2\in R_{\ell}, x_1,x_3\in R_r\): No monochromatic solution by Lemma \ref{integrality} (d) covering \(9\) cases.
      \item \(x_3\in R_{\ell}, x_1,x_2\in R_r\): No monochromatic solution by Lemma \ref{integrality} (c) covering \(9\) cases.
      \item \(x_2,x_3\in R_{\ell}, x_1\in R_r\): No monochromatic solution by Lemma \ref{integrality} (f) covering \(3\) cases.
      \item \(x_1,x_3\in R_{\ell}, x_2\in R_r\): No monochromatic solution by Lemma \ref{integrality} (g) covering \(3\) cases.
      \item \(x_1,x_2\in R_{\ell}, x_3\in R_r\): No monochromatic solution by Lemma \ref{lem:12ell-3r} covering \(3\) cases.
      \item \(x_1,x_2,x_3\in R_{r}\): No monochromatic solution by Lemma \ref{lem:R_r} covering \(27\) cases.
      }
    Hence, there exists no monochromatic solution to \(ax_1+ax_2=(a+1)x_3\) in colour \(1\).
  \end{proof}
\end{lemma}

\begin{lemma}\label{lem:B_r}
  For an odd positive integer \(a\geq 7\), the set \(B_r\) contains no solution to \({\cal E}(3,0; a,a,a+1)\).
  \begin{proof}
  Given \(x_1,x_2\in B_r\), assume that \(x_3={ax_1+ax_2 \over a+1}\) and \(x_3\in B_r\). Let
  \(x_1=k_1a^2\) and \(x_2=k_2a^2\) for positive integers \(k_1\) and \(k_2\). Substituting for \(x_3\), 
  we get the numerator to be \(a^3(k_1+k_2)\). Since $a^3$ shares no factor with $a+1$, if
  $a+1$ does not divide $k_1+k_2$, then $z\not\in {\mathbb Z}$, a contradiction.  
  Assume that $k_1+k_2=m(a+1)$ for some positive integer $m$. Then, $z=a^3m$ and assume $z\in B_r$. 
  For $1\leq i\leq a-1$, $a^4+ia^2$ is not an integer multiple of $a^3$.
  Since \(\min(B_r)=a^3\), we have $k_1+k_2\geq 2a$ and the smallest integer value of \(m\) is obtained 
  when \(k_1+k_2=(a+1)a\). This implies \(z\geq a^4\), but \(\max(B_2\cup B_3) < a^4\). 
  Hence, $z\not\in B_r$.
  
  Therefore,  the set $B_r$ contains no solution to \(ax_1+ax_2=(a+1)x_3\).
  \end{proof}
  \end{lemma}

  \begin{lemma}\label{lem:12ell-3Br}
    For odd positive integer \(a \geq 7\), there exists no monochromatic solution to \(ax_1+ax_2=(a+1)x_3\) if \(x_1,x_2 \in B_{\ell}\) and \(x_3\in B_r\).
  \begin{proof}
    Similar to the proof of Lemma \ref{lem:12ell-3r}.
  \end{proof}
  \end{lemma}

\begin{lemma}
  For odd positive integer \(a \geq 7\), there exists no monochromatic solution to \(ax_1+ax_2=(a+1)x_3\) in colour \(2\).
  \begin{proof}
    Let \(B_r = B_1\cup B_2\cup B_3\) with partition blocks \(B_1=\bigcup_{i=1}^{a-1}\set{a^4+i a^2}\), \(B_2=\bigcup_{i=1}^{a-1}\set{i a^3+j a^2: 0 \leq j\leq i-1}\), and \(B_3=\bigcup_{i=(a+1)/2}^{a-1}\set{i a^3+i a^2}\).
    There are \(4^3=64\) ways to select an integer for the variables \(x_1,x_2,x_3\) from the \(4\) sets \(B_{\ell}, B_1, B_2, B_3\). 
    Most of these cases can be analyzed directly using Lemma \ref{integrality} as follows:
    \myitemize{11pt}{
      \item \(x_1,x_2,x_3\in B_{\ell}\): No monochromatic solution by Lemma \ref{integrality} (a) covering \(1\) case.
      \item \(x_1\in B_{\ell}, x_2,x_3\in B_r\): No monochromatic solution by Lemma \ref{integrality} (e) covering \(9\) cases.
      \item \(x_2\in B_{\ell}, x_1,x_3\in B_r\): No monochromatic solution by Lemma \ref{integrality} (d) covering \(9\) cases.
      \item \(x_3\in B_{\ell}, x_1,x_2\in B_r\): No monochromatic solution by Lemma \ref{integrality} (c) covering \(9\) cases.
      \item \(x_2,x_3\in B_{\ell}, x_1\in B_r\): No monochromatic solution by Lemma \ref{integrality} (f) covering \(3\) cases.
      \item \(x_1,x_3\in B_{\ell}, x_2\in B_r\): No monochromatic solution by Lemma \ref{integrality} (g) covering \(3\) cases.
      \item \(x_1,x_2\in B_{\ell}, x_3\in B_r\): No monochromatic solution by Lemma \ref{lem:12ell-3Br} covering \(3\) cases.
      \item \(x_1,x_2,x_3\in B_{r}\): No monochromatic solution by Lemma \ref{lem:B_r} covering \(27\) cases.
      }
    Hence, there exists no monochromatic solution to \(ax_1+ax_2=(a+1)x_3\) in colour \(2\).
  \end{proof}
\end{lemma}

\section{SMT Instance via the \ZT\ Python API for Example \ref{example:z3}}\label{appendix:z3}

\begin{verbatim}
from z3 import *

# Declare symbolic variables
a = Int('a')
x1, x2, x3 = Ints('x_1 x_2 x_3')
k1, k2 = Ints('k_1 k_2')
i, j, k = Ints('i j k')

# Create the solver
s = Solver()

# Global assumptions
s.add(a >= 1)

# Equation constraint
s.add(a * x1 + a * x2 == (a + 1) * x3)

# Divisibility: a | x1
s.add(And(k1>0, x1 == a * k1))

# Non-divisibility: not(a^2 | x2)
s.add(ForAll([k2], And(k2>0, x2 != a**2 * k2)))

# Format expression for x3
s.add(x3 == a*(a + 1)*i + a*j + k)

# Symbolic format bounds
s.add(i >= 0, i < a**2)
s.add(j >= 0, j <= a)
s.add(k >= 2*(j // 2) + 1, k <= a - 1)

# Bounds for x1, x2, x3
s.add(x1 >= 1, x1 <= a**2)
s.add(x2 >= 1, x2 <= a**2)
s.add(x3 >= 1, x3 <= a**2)
\end{verbatim}

\end{document}

%% file: table1.tex
\begin{table*}\scriptsize
\begin{tabular}{|p{0.25cm}|p{0.55cm}|p{0.55cm}|p{0.55cm}|p{0.55cm}|p{0.55cm}|p{0.55cm}|p{0.55cm}|p{0.55cm}|p{0.55cm}|p{0.55cm}|p{0.55cm}|p{0.55cm}|p{0.55cm}|p{0.55cm}|p{0.55cm}|}
\hline
{$a$\textbackslash{$b$}} & 1 & 2 & 3 & 4 & 5 & 6 & 7 & 8 & 9 & 10 & 11 & 12 & 13 & 14 & 15 \\ \hline

\hline

1 & 14 & 14 & 27 & 64 & 125 & 216 & 343 & 512 & 729 & 1000 & 1331 & 1728 & 2197 & 2744 & 3375\\

2 & 43 & \ur{14} & 31 & \ur{14} & 125 & \ur{27} & 343 & \ur{64} & 729 & \ur{125} & 1331 & \ur{216} & 2197 & \ur{343} & 3375\\

3 & 94 & 61 & \ur{14} & 73 & 125 & \ur{14} & 343 & 512 & \ur{27} & 1000 & 1331 & \ur{64} & 2197 & 2744 & \ur{125}\\

4 & 173 & \ur{43} & 109 & \ur{14} & 141 & \ur{31} & 343 & \ur{14} & 729 & \ur{125} & 1331 & \ur{27} & 2197 & \ur{343} & 3375\\

5 & 286 & 181 & 186 & 180 & \ur{14} & 241 & 343 & 512 & 729 & \ur{14} & 1331 & 1728 & 2197 & 2744 & \ur{27}\\

6 & 439 & \ur{94} & \ur{43} & \ur{61} & 300 & \ur{14} & 379 & \ur{73} & \ur{31} & \ur{125} & 1331 & \ur{14} & 2197 & \ur{343} & \ur{125}\\

7 & 638 & 428 & 442 & 456 & 470 & 462 & \ur{14} & 561 & 729 & 1000 & 1331 & 1728 & 2197 & \ur{14} & 3375\\

8 & 889 & \ur{173} & 633 & \ur{43} & 665 & \ur{109} & 644 & \ur{14} & 793 & \ur{141} & 1331 & \ur{31} & 2197 & \ur{343} & 3375\\

9 & 1198 & 856 & \ur{94} & 892 & 910 & \ur{61} & 896 & 896 & \ur{14} & 1081 & 1331 & \ur{73} & 2197 & 2744 & \ur{125}\\

10 & 1571 & \ur{286} & 1171 & \ur{181} & \ur{43} & \ur{186} & 1190 & \ur{180} & 1206 & \ur{14} & 1431 & \ur{241} & 2197 & \ur{343} & \ur{31}\\

11 & 2014 & 1508 & 1530 & 1552 & 1574 & 1596 & 1618 & 1584 & 1575 & 1580 & \ur{14} & 1849 & 2197 & 2744 & 3375\\

12 & 2533 & \ur{439} & \ur{173} & \ur{94} & 2005 & \ur{43} & 2053 & \ur{61} & \ur{109} & \ur{300} & 2024 & \ur{14} & 2341 & \ur{379} & \ur{141}\\

13 & 3134 & 2432 & 2458 & 2484 & 2510 & 2536 & 2562 & 2588 & 2574 & 2530 & 2541 & 2544 & \ur{14} & 2913 & 3375\\

14 & 3823 & \ur{638} & 3039 & \ur{428} & 3095 & \ur{442} & \ur{43} & \ur{456} & 3207 & \ur{470} & 3113 & \ur{462} & 3146 & \ur{14} & 3571\\

15 & 4606 & 3676 & \ur{286} & 3736 & \ur{94} & \ur{181} & 3826 & 3856 & \ur{186} & \ur{61} & 3795 & \ur{180} & 3835 & 3836 & \ur{14}\\

16 & 5489 & \ur{889} & \rb{4465} & \ur{173} & \rb{4529} & \ur{633} & \rb{4593} & \ur{43} & \rb{4657} & \ur{665} & \rb{4576} & \ur{109} & \rb{4602} & \ur{644} & \rb{4620}\\

17 & 6478 & \rb{5288} & \rb{5322} & \rb{5356} & \rb{5390} & \rb{5424} & \rb{5458} & \rb{5492} & \rb{5526} & \rb{5560} & \rb{5594} & \rb{5424} & \rb{5447} & \rb{5474} & \rb{5475}\\

18 & 7579 & \ur{1198} & \ur{439} & \ur{856} & \rb{6355} & \ur{94} & \rb{6427} & \ur{892} & \ur{43} & \ur{910} & \rb{6571} & \ur{61} & \rb{6409} & \ur{896} & \ur{300}\\

19 & 8798 & \rb{7316} & \rb{7354} & \rb{7392} & \rb{7430} & \rb{7468} & \rb{7506} & \rb{7544} & \rb{7582} & \rb{7620} & \rb{7658} & \rb{7696} & \rb{7488} & \rb{7518} & \rb{7530}\\

20 & 10141 & \ur{1571} & \rb{8541} & \ur{286} & \ur{173} & \ur{1171} & \rb{8701} & \ur{181} & \rb{8781} & \ur{43} & \rb{8861} & \ur{186} & \rb{8941} & \ur{1190} & \ur{109}\\

21 & 11614 & \rb{9808} & \ur{638} & \rb{9892} & \rb{9934} & \ur{428} & \ur{94} & \rb{10060} & \ur{442} & \rb{10144} & \rb{10186} & \ur{456} & \rb{10270} & \ur{61} & \ur{470}\\

22 & 13223 & \ur{2014} & \rb{11287} & \ur{1508} & \rb{11375} & \ur{1530} & \rb{11463} & \ur{1552} & \rb{11551} & \ur{1574} & \ur{43} & \ur{1596} & \rb{11727} & \ur{1618} & \rb{11535}\\

23 & 14974 & \rb{12812} & \rb{12858} & \rb{12904} & \rb{12950} & \rb{12996} & \rb{13042} & \rb{13088} & \rb{13134} & \rb{13180} & \rb{13226} & \rb{13272} & \rb{13318} & \rb{13364} & \rb{13110}\\

24 & 16873 & \ur{2533} & \ur{889} & \ur{439} & \rb{14665} & \ur{173} & \rb{14761} & \ur{94} & \ur{633} & \ur{2005} & \rb{14953} & \ur{43} & \rb{15049} & \ur{2053} & \ur{665}\\

25 & 18926 & \rb{16376} & \rb{16426} & \rb{16476} & \ur{286} & \rb{16576} & \rb{16626} & \rb{16676} & \rb{16726} & \ur{181} & \rb{16826} & \rb{16876} & \rb{16926} & \rb{16976} & \ur{186}\\

26 & 21139 & \ur{3134} & \rb{18435} & \ur{2432} & \rb{18539} & \ur{2458} & \rb{18643} & \ur{2484} & \rb{18747} & \ur{2510} & \rb{18851} & \ur{2536} & \ur{43} & \ur{2562} & \rb{19059}\\

27 & 23518 & \rb{20548} & \ur{1198} & \rb{20656} & \rb{20710} & \ur{856} & \rb{20818} & \rb{20872} & \ur{94} & \rb{20980} & \rb{21034} & \ur{892} & \rb{21142} & \rb{21196} & \ur{910}\\

28 & 26069 & \ur{3823} & \rb{22933} & \ur{638} & \rb{23045} & \ur{3039} & \ur{173} & \ur{428} & \rb{23269} & \ur{3095} & \rb{23381} & \ur{442} & \rb{23493} & \ur{43} & \rb{23605}\\

29 & 28798 & \rb{25376} & \rb{25434} & \rb{25492} & \rb{25550} & \rb{25608} & \rb{25666} & \rb{25724} & \rb{25782} & \rb{25840} & \rb{25898} & \rb{25956} & \rb{26014} & \rb{26072} & \rb{26130}\\

30 & 31711 & \ur{4606} & \ur{1571} & \ur{3676} & \ur{439} & \ur{286} & \rb{28351} & \ur{3736} & \ur{1171} & \ur{94} & \rb{28591} & \ur{181} & \rb{28711} & \ur{3826} & \ur{43}\\

\hline
\end{tabular}
\caption{$R_3({\cal E}(3, 0;a, b, b))$ for $1 \leq b \leq 15$ and $1\leq a\leq 30$.
The previously unknown values are presented in boldface, and
the underlined entries correspond to equations whose coefficients are not coprime.} \label{tab:abb}
\end{table*}

%% file: table2.tex
\begin{table*}\scriptsize
\begin{tabular}{|p{0.25cm}|p{0.35cm}|p{0.35cm}|p{0.35cm}|p{0.35cm}|p{0.65cm}|p{0.65cm}|p{0.65cm}|p{0.65cm}|p{0.65cm}|p{0.65cm}|p{0.65cm}|p{0.65cm}|p{0.65cm}|p{0.65cm}|p{0.65cm}|}

\hline
{$b$\textbackslash{$a$}} & 1 & 2 & 3 & 4 & 5 & 6 & 7 & 8 & 9 & 10 & 11 & 12 & 13 & 14 & 15 \\ \hline

\hline

1 & 14 & \infinity & \infinity & \infinity & \infinity & \infinity & \infinity & \infinity & \infinity & \infinity & \infinity & \infinity & \infinity & \infinity & \infinity \\

2 & 1 & \ur{14} & 243 & \ur{\infinity} & \infinity & \ur{\infinity} & \infinity & \ur{\infinity} & \infinity & \ur{\infinity} & \infinity & \ur{\infinity} & \infinity & \ur{\infinity} & \infinity \\

3 & 54 & 54 & \ur{14} & 384 & 2000 & \ur{\infinity} & \infinity & \infinity & \ur{\infinity} & \infinity & \infinity & \ur{\infinity} & \infinity & \infinity & \ur{\infinity} \\

4 & \infinity & \ur{1} & 108 & \ur{14} & 875 & \ur{243} & 4459 & \ur{\infinity} & \infinity & \ur{\infinity} & \infinity & \ur{\infinity} & \infinity & \ur{\infinity} & \infinity \\

5 & \infinity & 105 & 135 & 180 & \ur{14} & 864 & 3430 & 3072 & 12393 & \ur{\infinity} & \infinity & \infinity & \infinity & \infinity & \ur{\infinity} \\

6 & \infinity & \ur{54} & \ur{1} & \ur{54} & 750 & \ur{14} & 3087 & \ur{384} & \ur{243} & \ur{2000} & \rb{27951} & \ur{\infinity} & \infinity & \ur{\infinity} & \ur{\infinity} \\

7 & \infinity & 455 & 336 & 308 & 875 & 756 & \ur{14} & 1536 & 8748 & 7500 & \rb{23958} & \rb{10368} & \rb{54925} & \ur{\infinity} & \infinity \\

8 & \infinity & \ur{\infinity} & 432 & \ur{1} & 1000 & \ur{108} & 2744 & \ur{14} & 8019 & \ur{875} & \rb{21296} & \ur{243} & \rb{48334} & \ur{4459} & \rb{97875} \\

9 & \infinity & \infinity & \ur{54} & 585 & 1125 & \ur{54} & 3087 & 1224 & \ur{14} & 6000 & \rb{18634} & \ur{384} & \rb{41743} & \rb{30184} & \ur{2000} \\

10 & \infinity & \ur{\infinity} & 1125 & \ur{105} & \ur{1} & \ur{135} & 3430 & \ur{180} & 7290 & \ur{14} & \rb{17303} & \ur{864} & \rb{37349} & \ur{3430} & \ur{243} \\

11 & \infinity & \infinity & 2019 & 847 & 1958 & 1188 & \rb{3773} & \rb{1672} & \rb{8019} & \rb{5500} & \ur{14} & \rb{6048} & \rb{35152} & \rb{24696} & \rb{77625} \\

12 & \infinity & \ur{\infinity} & \ur{\infinity} & \ur{54} & 2400 & \ur{1} & \rb{4116} & \ur{54} & \ur{108} & \ur{750} & \rb{15972} & \ur{14} & \rb{32955} & \ur{3087} & \ur{875} \\

13 & \infinity & \infinity & \infinity & 1710 & 3445 & 1963 & \rb{4459} & \rb{1456} & \rb{9477} & \rb{6500} & \rb{17303} & \rb{5616} & \ur{14} & \rb{21952} & \rb{60750} \\

14 & \infinity & \ur{\infinity} & \infinity & \ur{455} & 3675 & \ur{336} & \ur{1} & \ur{308} & \rb{10206} & \ur{875} & \rb{18634} & \ur{756} & \rb{30758} & \ur{14} & \rb{57375} \\

15 & \infinity & \infinity & \ur{\infinity} & 5408 & \ur{54} & \ur{105} & \rb{6615} & \rb{2760} & \ur{135} & \ur{54} & \rb{19965} & \ur{180} & \rb{32955} & \rb{20580} & \ur{14} \\

16 & \infinity & \ur{\infinity} & \infinity & \ur{\infinity} & 5725 & \ur{432} & \rb{7616} & \ur{1} & \rb{11664} & \ur{1000} & \rb{21296} & \ur{108} & \rb{35152} & \ur{2744} & \rb{54000}\\

17 & \infinity & \infinity & \infinity & \infinity & 8330 & 4743 & \rb{10064} & \rb{3825} & \rb{12393} & \rb{8500} & \rb{22627} & \rb{7344} & \rb{37349} & \rb{23324} & \rb{57375} \\

18 & \infinity & \ur{\infinity} & \ur{\infinity} & \ur{\infinity} & 12069 & \ur{54}  & \rb{10962} & \ur{585} & \ur{1} & \ur{1125} & \rb{23958} & \ur{54} & \rb{39546} & \ur{3087} & \ur{750}\\

19 & \infinity & \infinity & \infinity & \infinity & 16397 & 6726 & \rb{14782} & \rb{4332} & \rb{16853} & \rb{9500} & \rb{25289} & \rb{8208} & \rb{41743} & \rb{26068} & \rb{60750} \\

20 & \infinity & \ur{\infinity} & \infinity & \ur{\infinity} & \ur{\infinity} & \ur{1125} & \rb{14700} & \ur{105} & \rb{19080} & \ur{1} & \rb{26620} & \ur{135} & \rb{43940} & \ur{3430} & \ur{108}\\

21 & \infinity & \infinity & \ur{\infinity} & \infinity & \infinity & \ur{455} & \ur{54} & \rb{6699} & \ur{336} & \rb{12579} & \rb{27951} & \ur{308} & \rb{46137} & \ur{54} & \ur{875}\\

22 & \infinity & \ur{\infinity} & \infinity & \ur{\infinity} & \infinity & \ub{2019} & \rb{20580} & \ub{847} & \rb{25047} & \ub{1958} & \ur{1} & \ub{1188} & \rb{48334} & \ub{3773} & \rb{74250} \\

23 & \infinity & \infinity & \infinity & \infinity & \infinity & \rb{19056} & \rb{27853} & \rb{8556} & \rb{32453} & \rb{17250} & \rb{35926} & \rb{9936} & \rb{50531} & \rb{31556} & \rb{77625} \\

24 & \infinity & \ur{\infinity} & \ur{\infinity} & \ur{\infinity} & \infinity & \ur{\infinity} & \rb{28956} & \ur{54} & \ur{432} & \ub{2400} & \rb{39600} & \ur{1} & \rb{52728} & \ub{4116} & \ur{1000}\\

25 & \infinity & \infinity & \infinity & \infinity & \ur{\infinity} & \infinity & \rb{40163} & \rb{10200} & \rb{42975} & \ur{105} & \rb{47850} & \rb{12850} & \rb{54925} & \rb{34300} & \ur{135}\\

\hline 

\end{tabular}
\caption{$R_3({\cal E}(3, 0;a, a, b))$ for $1 \leq a\leq 15$ and $1\leq b\leq 25$.
The previously unknown values are presented in boldface, and
the underlined entries correspond to equations whose coefficients are not coprime.
}\label{tab:aab}
\end{table*}

%% file: arxiv.bbl
\begin{thebibliography}{26}
\expandafter\ifx\csname natexlab\endcsname\relax\def\natexlab#1{#1}\fi
\providecommand{\url}[1]{\texttt{#1}}
\providecommand{\href}[2]{#2}
\providecommand{\path}[1]{#1}
\providecommand{\DOIprefix}{doi:}
\providecommand{\ArXivprefix}{arXiv:}
\providecommand{\URLprefix}{URL: }
\providecommand{\Pubmedprefix}{pmid:}
\providecommand{\doi}[1]{\href{http://dx.doi.org/#1}{\path{#1}}}
\providecommand{\Pubmed}[1]{\href{pmid:#1}{\path{#1}}}
\providecommand{\bibinfo}[2]{#2}
\ifx\xfnm\relax \def\xfnm[#1]{\unskip,\space#1}\fi
\bibitem[{Rado(1933)}]{rado1933}
\bibinfo{author}{R.~Rado},
\newblock \bibinfo{title}{Studien zur {Kombinatorik}},
\newblock \bibinfo{journal}{Mathematische Zeitschrift} \bibinfo{volume}{36}
  (\bibinfo{year}{1933}) \bibinfo{pages}{242--280}.
  \DOIprefix\doi{10.1007/BF01188632}.
\bibitem[{Schaal(1995)}]{schaal1995}
\bibinfo{author}{D.~Schaal},
\newblock \bibinfo{title}{A family of 3-color {Rado} numbers},
\newblock \bibinfo{journal}{Congressus Numerantium} \bibinfo{volume}{111}
  (\bibinfo{year}{1995}) \bibinfo{pages}{150--160}.
\bibitem[{Adhikari et~al.(2016)Adhikari, Boza, Eliahou, Marin, Revuelta, and
  Sanz}]{adhikari2015}
\bibinfo{author}{S.~D. Adhikari}, \bibinfo{author}{L.~Boza},
  \bibinfo{author}{S.~Eliahou}, \bibinfo{author}{J.~M. Marin},
  \bibinfo{author}{M.~P. Revuelta}, \bibinfo{author}{M.~I. Sanz},
\newblock \bibinfo{title}{On the \(n\)-color {Rado} number for the equation
  \(x_1+x_2+\cdots+x_k+c=x_{k+1}\)},
\newblock \bibinfo{journal}{Mathematics of Computation} \bibinfo{volume}{85}
  (\bibinfo{year}{2016}) \bibinfo{pages}{2047--2064}.
  \DOIprefix\doi{10.1090/mcom3034}.
\bibitem[{Chang et~al.(2022)Chang, De~Loera, and Wesley}]{chang2022}
\bibinfo{author}{Y.~Chang}, \bibinfo{author}{J.~A. De~Loera},
  \bibinfo{author}{W.~J. Wesley},
\newblock \bibinfo{title}{Rado numbers and {SAT} computations},
\newblock in: \bibinfo{booktitle}{Proceedings of the 2022 International
  Symposium on Symbolic and Algebraic Computation}, ISSAC ’22,
  \bibinfo{publisher}{ACM}, \bibinfo{year}{2022}, p.
  \bibinfo{pages}{333–342}. \DOIprefix\doi{10.1145/3476446.3535494}.
\bibitem[{Myers(2015)}]{myers2015}
\bibinfo{author}{K.~Myers}, \bibinfo{title}{Computational advances in {Rado}
  numbers}, Ph.D. thesis, Rutgers The State University of New Jersey - New
  Brunswick, \bibinfo{year}{2015}. \DOIprefix\doi{10.7282/T3GH9KTT}.
\bibitem[{Meurer et~al.(2017)Meurer, Smith, Paprocki, \v{C}ert\'{i}k,
  Kirpichev, Rocklin, Kumar, Ivanov, Moore, Singh, Rathnayake, Vig, Granger,
  Muller, Bonazzi, Gupta, Vats, Johansson, Pedregosa, Curry, Terrel,
  Rou\v{c}ka, Saboo, Fernando, Kulal, Cimrman, and Scopatz}]{sympy}
\bibinfo{author}{A.~Meurer}, \bibinfo{author}{C.~P. Smith},
  \bibinfo{author}{M.~Paprocki}, \bibinfo{author}{O.~\v{C}ert\'{i}k},
  \bibinfo{author}{S.~B. Kirpichev}, \bibinfo{author}{M.~Rocklin},
  \bibinfo{author}{A.~Kumar}, \bibinfo{author}{S.~Ivanov},
  \bibinfo{author}{J.~K. Moore}, \bibinfo{author}{S.~Singh},
  \bibinfo{author}{T.~Rathnayake}, \bibinfo{author}{S.~Vig},
  \bibinfo{author}{B.~E. Granger}, \bibinfo{author}{R.~P. Muller},
  \bibinfo{author}{F.~Bonazzi}, \bibinfo{author}{H.~Gupta},
  \bibinfo{author}{S.~Vats}, \bibinfo{author}{F.~Johansson},
  \bibinfo{author}{F.~Pedregosa}, \bibinfo{author}{M.~J. Curry},
  \bibinfo{author}{A.~R. Terrel}, \bibinfo{author}{v.~Rou\v{c}ka},
  \bibinfo{author}{A.~Saboo}, \bibinfo{author}{I.~Fernando},
  \bibinfo{author}{S.~Kulal}, \bibinfo{author}{R.~Cimrman},
  \bibinfo{author}{A.~Scopatz},
\newblock \bibinfo{title}{{SymPy}: Symbolic computing in {Python}},
\newblock \bibinfo{journal}{PeerJ Computer Science} \bibinfo{volume}{3}
  (\bibinfo{year}{2017}) \bibinfo{pages}{e103}.
  \DOIprefix\doi{10.7717/peerj-cs.103}.
\bibitem[{{de Moura} and Bjørner(2008)}]{deMoura2008}
\bibinfo{author}{L.~{de Moura}}, \bibinfo{author}{N.~Bjørner},
  \bibinfo{title}{Z3: An Efficient {SMT} Solver}, \bibinfo{publisher}{Springer
  Berlin Heidelberg}, \bibinfo{year}{2008}, p. \bibinfo{pages}{337–340}.
  \DOIprefix\doi{10.1007/978-3-540-78800-3_24}.
\bibitem[{Ábrahám et~al.(2016)Ábrahám, Abbott, Becker, Bigatti, Brain,
  Buchberger, Cimatti, Davenport, England, Fontaine, Forrest, Griggio,
  Kroening, Seiler, and Sturm}]{abraham2016}
\bibinfo{author}{E.~Ábrahám}, \bibinfo{author}{J.~Abbott},
  \bibinfo{author}{B.~Becker}, \bibinfo{author}{A.~M. Bigatti},
  \bibinfo{author}{M.~Brain}, \bibinfo{author}{B.~Buchberger},
  \bibinfo{author}{A.~Cimatti}, \bibinfo{author}{J.~H. Davenport},
  \bibinfo{author}{M.~England}, \bibinfo{author}{P.~Fontaine},
  \bibinfo{author}{S.~Forrest}, \bibinfo{author}{A.~Griggio},
  \bibinfo{author}{D.~Kroening}, \bibinfo{author}{W.~M. Seiler},
  \bibinfo{author}{T.~Sturm}, \bibinfo{title}{$\mathsf{SC}^\mathsf{2}$:
  Satisfiability Checking Meets Symbolic Computation (Project Paper)},
  \bibinfo{publisher}{Springer International Publishing}, \bibinfo{year}{2016},
  p. \bibinfo{pages}{28–43}. \DOIprefix\doi{10.1007/978-3-319-42547-4_3}.
\bibitem[{Bright et~al.(2022)Bright, Kotsireas, and Ganesh}]{Bright2022}
\bibinfo{author}{C.~Bright}, \bibinfo{author}{I.~Kotsireas},
  \bibinfo{author}{V.~Ganesh},
\newblock \bibinfo{title}{When satisfiability solving meets symbolic
  computation},
\newblock \bibinfo{journal}{Communications of the ACM} \bibinfo{volume}{65}
  (\bibinfo{year}{2022}) \bibinfo{pages}{64–72}.
  \DOIprefix\doi{10.1145/3500921}.
\bibitem[{Bright et~al.(2021)Bright, Cheung, Stevens, Kotsireas, and
  Ganesh}]{Bright2021}
\bibinfo{author}{C.~Bright}, \bibinfo{author}{K.~K.~H. Cheung},
  \bibinfo{author}{B.~Stevens}, \bibinfo{author}{I.~Kotsireas},
  \bibinfo{author}{V.~Ganesh},
\newblock \bibinfo{title}{A {SAT}-based resolution of {Lam's} problem},
\newblock \bibinfo{journal}{Proceedings of the AAAI Conference on Artificial
  Intelligence} \bibinfo{volume}{35} (\bibinfo{year}{2021})
  \bibinfo{pages}{3669–3676}. \DOIprefix\doi{10.1609/aaai.v35i5.16483}.
\bibitem[{Kirchweger et~al.(2023)Kirchweger, Peitl, and
  Szeider}]{Kirchweger2023}
\bibinfo{author}{M.~Kirchweger}, \bibinfo{author}{T.~Peitl},
  \bibinfo{author}{S.~Szeider},
\newblock \bibinfo{title}{Co-certificate learning with {SAT} modulo
  symmetries},
\newblock in: \bibinfo{booktitle}{Proceedings of the Thirty-Second
  International Joint Conference on Artificial Intelligence}, IJCAI-2023,
  \bibinfo{publisher}{International Joint Conferences on Artificial
  Intelligence Organization}, \bibinfo{year}{2023}, p.
  \bibinfo{pages}{1944–1953}. \DOIprefix\doi{10.24963/ijcai.2023/216}.
\bibitem[{Li et~al.(2024)Li, Bright, and Ganesh}]{Li2024}
\bibinfo{author}{Z.~Li}, \bibinfo{author}{C.~Bright},
  \bibinfo{author}{V.~Ganesh},
\newblock \bibinfo{title}{A {SAT} solver + computer algebra attack on the
  minimum {Kochen–Specker} problem},
\newblock in: \bibinfo{booktitle}{Proceedings of the Thirty-Third International
  Joint Conference on Artificial Intelligence}, IJCAI-2024,
  \bibinfo{publisher}{International Joint Conferences on Artificial
  Intelligence Organization}, \bibinfo{year}{2024}, p.
  \bibinfo{pages}{1898–1906}. \DOIprefix\doi{10.24963/ijcai.2024/210}.
\bibitem[{Kouril and Paul(2008)}]{kouril2006}
\bibinfo{author}{M.~Kouril}, \bibinfo{author}{J.~L. Paul},
\newblock \bibinfo{title}{The van der {Waerden} number \( w(2,6) \) is 1132},
\newblock \bibinfo{journal}{Experimental Mathematics} \bibinfo{volume}{17}
  (\bibinfo{year}{2008}) \bibinfo{pages}{53--61}.
  \DOIprefix\doi{10.1080/10586458.2008.10129025}.
\bibitem[{Herwig et~al.(2007)Herwig, Heule, van Lambalgen, and van
  Maaren}]{herwig2007}
\bibinfo{author}{P.~R. Herwig}, \bibinfo{author}{M.~J.~H. Heule},
  \bibinfo{author}{P.~M. van Lambalgen}, \bibinfo{author}{H.~van Maaren},
\newblock \bibinfo{title}{A new method to construct lower bounds for van der
  {Waerden} numbers},
\newblock \bibinfo{journal}{The Electronic Journal of Combinatorics}
  \bibinfo{volume}{14} (\bibinfo{year}{2007}). \DOIprefix\doi{10.37236/925}.
\bibitem[{Ahmed(2009)}]{ahmed2009}
\bibinfo{author}{T.~Ahmed},
\newblock \bibinfo{title}{Some new van der {Waerden} numbers and some van der
  {Waerden-type} numbers},
\newblock \bibinfo{journal}{Integers} \bibinfo{volume}{9}
  (\bibinfo{year}{2009}) \bibinfo{pages}{65--76}.
  \DOIprefix\doi{10.1515/INTEG.2009.007}.
\bibitem[{Ahmed(2010)}]{ahmed2010}
\bibinfo{author}{T.~Ahmed},
\newblock \bibinfo{title}{Two new van der {Waerden} numbers: \(w(2; 3, 17)\)
  and \(w(2; 3, 18)\)},
\newblock \bibinfo{journal}{Integers} \bibinfo{volume}{10}
  (\bibinfo{year}{2010}) \bibinfo{pages}{369--377}.
  \DOIprefix\doi{10.1515/integ.2010.032}.
\bibitem[{Ahmed(2012)}]{ahmed2011}
\bibinfo{author}{T.~Ahmed},
\newblock \bibinfo{title}{On computation of exact van der {Waerden} numbers},
\newblock \bibinfo{journal}{Integers} \bibinfo{volume}{12}
  (\bibinfo{year}{2012}) \bibinfo{pages}{417--425}.
  \DOIprefix\doi{10.1515/integ.2011.112}.
\bibitem[{Ahmed(2013)}]{ahmed2013}
\bibinfo{author}{T.~Ahmed},
\newblock \bibinfo{title}{Some more van der {Waerden} numbers},
\newblock \bibinfo{journal}{Journal of Integer Sequences} \bibinfo{volume}{16}
  (\bibinfo{year}{2013}). \bibinfo{note}{Article 13.4.4}.
\bibitem[{Ahmed et~al.(2014)Ahmed, Kullmann, and Snevily}]{aks2014}
\bibinfo{author}{T.~Ahmed}, \bibinfo{author}{O.~Kullmann},
  \bibinfo{author}{H.~Snevily},
\newblock \bibinfo{title}{On the van der {Waerden} numbers
  \(\operatorname{w}(2;3,t)\)},
\newblock \bibinfo{journal}{Discrete Applied Mathematics} \bibinfo{volume}{174}
  (\bibinfo{year}{2014}) \bibinfo{pages}{27--51}.
  \DOIprefix\doi{10.1016/j.dam.2014.05.007}.
\bibitem[{Ahmed and Schaal(2016)}]{as2015}
\bibinfo{author}{T.~Ahmed}, \bibinfo{author}{D.~J. Schaal},
\newblock \bibinfo{title}{On generalized {Schur} numbers},
\newblock \bibinfo{journal}{Experimental Mathematics} \bibinfo{volume}{25}
  (\bibinfo{year}{2016}) \bibinfo{pages}{213--218}.
  \DOIprefix\doi{10.1080/10586458.2015.1070776}.
\bibitem[{Ahmed et~al.(2023)Ahmed, Boza, Revuelta, and Sanz}]{abrs2023a}
\bibinfo{author}{T.~Ahmed}, \bibinfo{author}{L.~Boza}, \bibinfo{author}{M.~P.
  Revuelta}, \bibinfo{author}{M.~I. Sanz},
\newblock \bibinfo{title}{Exact values and lower bounds on the \(n\)-color weak
  {Schur} numbers for \(n=2,3\)},
\newblock \bibinfo{journal}{The Ramanujan Journal} \bibinfo{volume}{62}
  (\bibinfo{year}{2023}) \bibinfo{pages}{347--363}.
  \DOIprefix\doi{10.1007/s11139-023-00760-y}.
\bibitem[{Heule(2018)}]{heule2018}
\bibinfo{author}{M.~J.~H. Heule},
\newblock \bibinfo{title}{Schur number five},
\newblock in: \bibinfo{booktitle}{Proceedings of AAAI-18},
  \bibinfo{year}{2018}, pp. \bibinfo{pages}{6598--6606}.
  \DOIprefix\doi{10.1609/aaai.v32i1.12209}.
\bibitem[{Biere et~al.(2024{\natexlab{a}})Biere, Faller, Fazekas, Fleury,
  Froleyks, and Pollitt}]{biere2024cadical}
\bibinfo{author}{A.~Biere}, \bibinfo{author}{T.~Faller},
  \bibinfo{author}{K.~Fazekas}, \bibinfo{author}{M.~Fleury},
  \bibinfo{author}{N.~Froleyks}, \bibinfo{author}{F.~Pollitt},
\newblock \bibinfo{title}{{CaDiCaL 2.0}},
\newblock in: \bibinfo{editor}{A.~Gurfinkel}, \bibinfo{editor}{V.~Ganesh}
  (Eds.), \bibinfo{booktitle}{Computer Aided Verification - 36th International
  Conference, {CAV} 2024, Montreal, QC, Canada, July 24-27, 2024, Proceedings,
  Part {I}}, volume \bibinfo{volume}{14681} of \textit{\bibinfo{series}{Lecture
  Notes in Computer Science}}, \bibinfo{publisher}{Springer},
  \bibinfo{year}{2024}{\natexlab{a}}, pp. \bibinfo{pages}{133--152}.
  \DOIprefix\doi{10.1007/978-3-031-65627-9_7}.
\bibitem[{Biere et~al.(2024{\natexlab{b}})Biere, Faller, Fazekas, Fleury,
  Froleyks, and Pollitt}]{biere2024kissat}
\bibinfo{author}{A.~Biere}, \bibinfo{author}{T.~Faller},
  \bibinfo{author}{K.~Fazekas}, \bibinfo{author}{M.~Fleury},
  \bibinfo{author}{N.~Froleyks}, \bibinfo{author}{F.~Pollitt},
\newblock \bibinfo{title}{{CaDiCaL}, {Gimsatul}, {IsaSAT} and {Kissat} entering
  the {SAT Competition 2024}},
\newblock in: \bibinfo{editor}{M.~Heule}, \bibinfo{editor}{M.~Iser},
  \bibinfo{editor}{M.~J{\"a}rvisalo}, \bibinfo{editor}{M.~Suda} (Eds.),
  \bibinfo{booktitle}{Proc.~of {SAT Competition} 2024 -- Solver, Benchmark and
  Proof Checker Descriptions}, volume \bibinfo{volume}{B-2024-1} of
  \textit{\bibinfo{series}{Department of Computer Science Report Series B}},
  \bibinfo{publisher}{University of Helsinki},
  \bibinfo{year}{2024}{\natexlab{b}}, pp. \bibinfo{pages}{8--10}.
\bibitem[{Ignatiev et~al.(2018)Ignatiev, Morgado, and
  Marques{-}Silva}]{imms-sat18}
\bibinfo{author}{A.~Ignatiev}, \bibinfo{author}{A.~Morgado},
  \bibinfo{author}{J.~Marques{-}Silva},
\newblock \bibinfo{title}{{PySAT:} {A} {Python} toolkit for prototyping with
  {SAT} oracles},
\newblock in: \bibinfo{booktitle}{SAT}, \bibinfo{year}{2018}, pp.
  \bibinfo{pages}{428--437}. \DOIprefix\doi{10.1007/978-3-319-94144-8_26}.
\bibitem[{Shallit(2021)}]{shallit2021}
\bibinfo{author}{J.~Shallit},
\newblock \bibinfo{title}{Say no to case analysis: Automating the drudgery of
  case-based proofs},
\newblock in: \bibinfo{booktitle}{CIAA 2021}, volume \bibinfo{volume}{12803} of
  \textit{\bibinfo{series}{LNCS}}, \bibinfo{year}{2021}, pp.
  \bibinfo{pages}{14--24}. \DOIprefix\doi{10.1007/978-3-030-79121-6_2}.

\end{thebibliography}
